\documentclass[a4paper, UKenglish, 11pt, bibliography=totocnumbered]{article}

\usepackage[utf8]{inputenc}

\usepackage{amsmath}
\usepackage{amsthm}
\usepackage{thmtools}
\usepackage{mathtools}
\usepackage{amsfonts}
\usepackage{bbm}

\usepackage[hidelinks]{hyperref}
\usepackage[capitalize,noabbrev, nameinlink]{cleveref}

\newcommand{\footremember}[2]{%
   \footnote{#2}
    \newcounter{#1}
    \setcounter{#1}{\value{footnote}}%
}

\usepackage{lineno}

\usepackage[margin=2.56cm]{geometry}

\usepackage{titlesec}

\titleformat*{\paragraph}{\itshape}

\usepackage{float}
\usepackage{graphicx}
\usepackage{subcaption}
\usepackage[dvipsnames]{xcolor}

\usepackage{blkarray}
\usepackage{tikz}
\usepackage{tikz-cd}
\usetikzlibrary{arrows}

\newcommand{\DoTikzmark}[1]{%
  \tikz[remember picture] \coordinate[shift={(.6ex,-.4ex)}](#1);%
}
\newcommand{\colrow}[3][]{%
  \tikz[overlay,remember picture, line width=10pt]
    \draw[shorten >=-.15em, shorten <=-1em, #1] (#2)--(#3);
}

\pgfdeclarelayer{fg}    
\pgfsetlayers{main, fg}
\usepackage{rotating}
\usepackage{adjustbox}
\usepackage{colortbl}
\newcommand\x{\colorbox{black}{}}
\usepackage{multirow}

\newtheorem{theorem}{Theorem}[section]

\theoremstyle{definition}
\newtheorem{lemma}[theorem]{Lemma}

\newtheorem{example}[theorem]{Example}
\newtheorem{proposition}[theorem]{Proposition}
\newtheorem{point}[theorem]{}

\newtheorem*{lemma-non}{Lemma}

\newtheorem{mthm}{Main Theorem}

\counterwithin{subsection}{section}

\usepackage[ruled,vlined, linesnumbered]{algorithm2e}
\usepackage{algpseudocode}

\usepackage{float}
\usepackage{enumitem}

\usepackage{orcidlink}

\newcommand{\ignore}[1]{}

\newcommand{\lowp}[1]{low(#1)}

\newcommand{\Prob}[1]{\mathbb{P} \left( #1 \right)} 

\newcommand{\bdmat}{D}

\newcommand{\define}[1]{{\bf #1}} 
\usepackage[normalem]{ulem}

\newcommand{\fillin}{\#}
\newcommand{\row}{r}
\newcommand{\ver}{n}
\newcommand{\column}{c}
\newcommand{\betti}{\beta}
\newcommand{\indicator}{\mathbbm{1}}

\newcommand{\radius}{\rho} 

\newcommand{\cechmod}{\check{C}(\mathcal{X}_n,\alpha)}
\newcommand{\cechfilt}{\check{C}(\mathcal{X}_n)}
\newcommand{\VRcompl}{\text{VR}(\mathcal{X}_n,\alpha)}
\newcommand{\VRfilt}{\text{VR}(\mathcal{X}_n)}
\newcommand{\ERcompl}{\text{ER}(n,\alpha)}
\newcommand{\ERfilt}{\text{ER}(n)}

\newcommand{\Cspace}[0]{\mathfrak{C}(G)}
\newcommand{\Tspace}[0]{\mathfrak{T}(G)}
\newcommand{\nbtri}[1]{\tau(#1)}
\newcommand{\neighG}[1]{\mathcal{N}(#1)}
\newcommand{\neigh}[2]{\mathcal{N}_{#1}(#2)}
\newcommand{\cut}[1]{\nabla(#1)}
\newcommand{\spannedG}[2]{#1[#2]}
\newcommand{\fullG}[1]{K_{#1}}
\newcommand{\abs}[1]{\Bigr\vert #1 \Bigr\vert}

\usepackage[backend=biber, 
style=authoryear, 
maxbibnames=10,
natbib=true, 
hyperref=true, 
]{biblatex}

\renewcommand{\cite}{\parencite}

\renewbibmacro{in:}{}
\DeclareFieldFormat{pages}{#1}
\begin{document}

\title{Expected Complexity of Persistence Barcode Computation via Matrix Reduction}

\author{Barbara Giunti \footremember{SUNY}{Graz University of Technology and SUNY University at Albany, 1400 Washington Avenue, HD-125, \texttt{bgiunti@albany.edu},\orcidlink{0000-0002-3500-8286}}, Guillaume Houry \footremember{Paris}{{\'E}cole Polytechnique Palaiseau, Route de Saclay, 91128 Palaiseau Cedex, \texttt{guillaume.houry@live.fr}}, Michael Kerber \footremember{TU1}{Graz University of Technology, Kopernikusgasse 24, Graz, Austria, \texttt{kerber@tugraz.at}, \orcidlink{0000-0002-8030-9299}}, Matthias S\"ols \footremember{TU}{Graz University of Technology, Kopernikusgasse 24, Graz, Austria, \texttt{{matthias.soels@tum.de}}, Corresponding Author, \orcidlink{0009-0004-9124-9027}}}

\date{}

\maketitle

\begin{abstract}
We study the algorithmic complexity of computing the persistence barcode of a randomly generated filtration. 
We provide a general technique to bound the expected complexity of
reducing the boundary matrix in terms of the density of its reduced form.
We apply this technique finding upper bounds for the average fill-in (number of non-zero entries) of the boundary matrix on \v{C}ech, Vietoris--Rips and Erd\H{o}s--R\'enyi filtrations after matrix reduction, thus obtaining bounds on the expected complexity of the barcode computation.
Our method is based on previous results on the expected Betti numbers of the corresponding complexes. 
Our fill-in bounds for \v{C}ech and Vietoris--Rips complexes are asymptotically tight up to a logarithmic factor. 
In particular, both our fill-in and computation bounds are better than the worst-case estimates.
We also provide an Erd\H{o}s--R\'enyi filtration realising the worst-case fill-in and computation.
\end{abstract}



\section{Introduction}
\paragraph{Motivation and results.} 
The (persistence) barcode is an invariant that extracts topological information from data. 
It has been proven to be extremely useful in applications (see \cite{donut} for over 300 examples), not only for the insights it provides on the data, but also because it can be converted into formats amenable to statistical and Machine Learning analysis. 
Therefore, understanding its computation is crucial for data analysis.
The standard algorithm used to compute barcodes, first introduced in \cite{edelsbrunner2000topological}, is based on the Gaussian reduction of the boundary matrix.
It performs left-to-right column additions until the indexes of the lowest elements of non-zero columns in the matrix are pairwise distinct; the matrix is called \define{reduced} in this case.
For a $(\row\times \column)$-boundary matrix with $r\leq c$, this reduction process runs in $\mathcal{O}(\row^2\column)$ time, and this high complexity can be indeed achieved by concrete families of examples (see \cite{morozov2005persistence} or \cref{S_worst_case}).
However, designing these worst-case examples requires some care~-- for instance, the boundary matrix necessarily has to become dense (i.e., has $\Omega(r^2)$ non-zero entries) during the reduction. 
On the other hand, such dense reduced matrices are hardly formed in realistic data sets, and the reduction algorithm scales closer to linear in practice \cite{bauer2017phat, otter2017roadmap}. 
This leads to the hypothesis that the worst-case examples are somewhat pathological, and the ``typical'' performance of the algorithm is better than what the worst-case predicts.
The motivation of this paper is to provide formal evidence for this hypothesis, 
mathematically grounding the displayed efficiency of the barcode computation.

 Our method hinges on two links: the fact that the computational complexity is bounded above by the density of the reduced boundary matrix, and the fact that dense columns in the reduced boundary matrix can be related to non-trivial homology of the filtration.
Hence, if a random filtration is unlikely to have nonzero Betti number from a certain step value onward, then we can bound its expected computational complexity. In order to study the ``typical" performance of the matrix reduction algorithm, the following instances of random filtrations are considered:

\begin{description}
\item[\v{C}ech model.] Sample $\ver$ points i.i.d. uniformly distributed from the $d$-dimensional unit cube and build the \v{C}ech filtration over these points.   
\item[Vietoris--Rips model.] Construct the Vietoris--Rips filtration over with points placed as in the \v{C}ech case.
\item[Erd\H{o}s--R\'enyi model.] Given $\ver$ vertices, apply a random permutation on the $\binom{\ver}{2}$ edges, and build the clique filtration over this edge order. 
\end{description}

The resulting $k$-dimensional boundary matrices consist of $\row=\binom{\ver}{k+1}=\Theta(\ver^{k+1})$ rows and $\column=\binom{\ver}{k+2}=\Theta(\ver^{k+2})$ columns. 
The naive bound for the cost of matrix reduction would therefore yield a time complexity of $\mathcal{O}(n^{3k+4})$ which assumes that the boundary matrices gather $\Omega(n^{2k+2})$ non-zero entries arising during matrix reduction.

We refer to the number of non-zero entries of the reduced matrix
as the \define{fill-in}. 
Our first main result is that the expected fill-in is given by $\mathcal{O}(n^{2}\log^{2k}n+n^{k+1})$ and the expected cost of matrix reduction is bounded by $\mathcal{O}(n^{k+4}\log^{2k}n+ n^{2k+3})$, for all $k\geq 1$ for \v{C}ech and Vietoris--Rips filtrations. 
Our second main result is that for Erd\H{o}s--R\'enyi case in degree $k=\, 1$, the expected fill-in and cost of matrix reduction are bounded by $\mathcal{O}(n^{3}\log n)$ and $\mathcal{O}(n^{6}\log n)$, respectively.
Note that both the expected fill-in as well as the expected cost for all three models is asymptotically better than the worst-case prediction. 

In the \v{C}ech and Vietoris--Rips case, our bound on the fill-in is asymptotically tight for $k>1$ because $\Omega(n^{k+1})$ is a lower bound on the fill-in.
For $k=1$, the bound becomes $O(n^2\log^2 n)$ which matches the lower bound
of $\Omega(n^2)$ up to a logarithmic factor.
We also provide some experiments that suggest that neither our fill-in bound for
the Erd\H{o}s--R\'enyi case nor our time bound for the \v{C}ech and Vietoris--Rips case are tight. 
Moreover, we present a construction realizing the worst-case for the (clique) Erd\H{o}s--R\'enyi model, for which the reduction algorithm for $k=1$ yields a matrix with $\mathcal{O}(n^4)$ fill-in and runs in $\mathcal{O}(n^7)$ time. 
This shows that the worst-case bounds on fill-in and runtime are tight for this filtration type.

\paragraph{Proof outline.}
We illustrate the proof for the \v{C}ech model, the Vietoris--Rips case is proved similarly. The unreduced boundary matrix encodes the filtration arising from the nerves of a growing union of balls, based around the randomly sampled points. 
During matrix reduction, columns in the boundary matrix get added to each other from left-to-right whenever their (non-zero) lowest entries coincide. 
Three cases are possible: A column may remain unchanged during matrix reduction, a column may turn to a zero-column, or a column may get reduced to some new non-zero lowest entry. 
The last case is critical: columns that undergo  reduction may get filled-in, i.e., become dense. 
In turn, dense columns slow down computation, because they affect memory but especially since they considerably increase the amount of operations required.
Such columns typically represents a topological feature of non-zero persistence.  
By a result of Kahle \cite{kahle2011random}, having non-zero Betti numbers gets very unlikely after a certain scale in the filtration and a high number of sampled points. 
This implies that the number of such dense columns is bounded in expectation. 
The same approach also works for Erd\H{o}s--R\'enyi filtrations, adapting a probabilistic bound for Betti numbers from~\cite{demarco2013triangle}
to our situation. 

\paragraph{Conference version.}
Parts of this work have already appeared as a conference version \cite{conferenceversion}. 
The framework introduced in \cite{conferenceversion} is limited to clique-filtrations with boundary matrices of dimension one. 
Hence only the Erd\H{o}s--R\'enyi and Vietoris--Rips model were considered,
and the dimension was restricted to $1$. 
In the present work, we add the analysis of the \v{C}ech model (which is not a clique filtration) and extend the analysis of the Vietoris--Rips model to arbitrary dimensions.

\paragraph{Related work.}
There are many variants of the standard reduction algorithm with the goal to improve its practical performance, partially with tremendous speed-ups, e.g.~\cite{adams2014javaplex, bauer2017phat, henselman2016matroid, maria2014gudhi, morozov2007dionysus, perez2021giotto}.
Even faster algorithms exist for special types of filtrations, for instance Vietoris-Rips complex~\cite{bauer2021ripser},
or in combination with pre-processing metods for voxel data~\cite{robins-skeletonization,wagner-socg,tierny-sandwich}.
All these approaches are eventually based on matrix reduction and do not overcome the worst-case complexity of Gaussian elimination. 
We consider only the standard reduction algorithm in our analysis although we suspect that our techniques apply to many of these variants as well. 
An asymptotically faster algorithm in matrix-multiplication time is
known~\cite{zigzag_matrix_multiplication}, as well as a randomized output-sensitive algorithm that computes only the most persistent features~\cite{Chen2013Output}. 
However, these approaches are not based on elementary
column operations and slower in practice.

In the persistence computation, the order of the simplices (and thus of the columns and rows in the boundary matrix) is crucial and can be altered only in specific cases \cite{bauer2017phat, bauer2022keeping, abcw}. 
This order also determines which elements can be used as pivots.
For that reason, we did not see how to transfer analyses of related problems,
such as the expected complexity of computing the Smith normal form~\cite{smithnormal} or the study of fill-in for linear-algebraic algorithms~\cite{directmethods,order_pivoting} to our setup. 
These methods require either to interleave row and column operations and swap rows and columns, or to reorder the rows and columns. 

The only previous work on the average complexity of persistence computation is by Kerber and Schreiber~\cite{Kerber2020Expected,SchreiberThesis}.
They show that, for the so-called shuffled random model, the average complexity is better than what the worst-case predicts.
However, the shuffled model is further away from realistic (simplicial) inputs than the three models studied in this paper. 
Moreover, their analysis requires a special variant of the reduction algorithm while our analysis applies to the original reduction algorithm with no changes.
In addition, the PhD thesis of Schreiber~\cite{SchreiberThesis} contains extensive experimental evaluations of several random models (including Vietoris--Rips and Erd\H{o}s--R\'enyi); our experiments partially redo and confirm these evaluations.

While the computational complexity for persistence has hardly been studied in terms of expectation, extensive efforts have gone into expected topological properties
of random simplicial complexes. We refer to the surveys by Kahle~\cite{Kahle-survey} and Bobrowski and Kahle~\cite{BobrowskiKahle} for an overview for the general
and the geometric case, respectively.
From this body of literature, we use the works by Demarco, Hamm, and Kahn~\cite{demarco2013triangle} and Kahle~\cite{kahle2011random} in our work. 
More in detail, we adapt some of their results on the expected Betti numbers of random filtrations to upper-bound the number of computationally expensive columns, thus obtaining a bound on the expected computational complexity.
There are also recent efforts to study expected properties of persistent homology over random filtrations, for instance the expected length 
of the maximally persistent cycles in a uniform Poisson process~\cite{bobrowski2017maximally}, properties of the expected persistence diagram over random point clouds~\cite{Divol2019density}, or the expected number of intervals in the decomposition of multiparameter persistence modules \cite{angel_socg1}, \cite{angel_socg2}. 

Finally, to the best of our knowledge, the only construction to achieve the worst-case running time for matrix reduction is by Morozov~\cite{morozov2005persistence}, which however involves only a linear number of edges and triangles with respect to the number of vertices. 
Therefore, it is not one of the models we considered, which are more common in data analysis.
\paragraph{Outline.}
In \cref{S_basics}, we introduce basic notions on (boundary) matrices and their reduction as well as simplicial homology.
In \cref{Fill-in analysis}, we prove the connection between Betti numbers and certain columns of the reduced matrix, which leads to a generic bound for the fill-in.
We then apply the general bound for the \v{C}ech (\cref{section_cech}), Vietoris--Rips (\cref{section_rips}) and Erd\H{o}s--R\'enyi (\cref{section_erdos_renyi}) filtrations. 
The technical results of \cref{section_cech,section_rips,section_erdos_renyi} are proven in \cref{A_lemma_C,A_lemma_VR,A_lemma_ER}, respectively.
In \cref{S_experiments}, we compare our bounds with experimental evaluation. 
In \cref{S_worst_case}, we construct a clique filtration realising the worst-case fill-in and cost.
We conclude in Section~\ref{section_conclusion}. 

\section{Basic notions}\label{S_basics}
\subsection{Matrix reduction}\label{subsec:Matrix_reduction}
In the following, fix an $(\row\times\column)$-matrix $M$ over $\mathbb{Z}_2$, the field with two elements,
and let its columns be denoted by $M_1,\ldots,M_\column$.
For a non-zero column $M_i$, we let its \define{pivot} be the index of the lowest row
in the matrix that has a non-zero entry, denoted by $\lowp{M_i}$. 
We write $\fillin M_i$ for the number of non-zero entries in the column,
and $\fillin M\coloneqq\sum_{i=1}^\column \fillin M_i$ for the number of non-zero entries
in the matrix. Clearly, $\fillin M\leq \row\cdot\column$; if $\fillin M$ is significantly smaller than that value (e.g., only linear in $\column$),
then the matrix is usually called ``sparse''. 

A \define{left-to-right column addition} is the operation
of replacing $M_i$ with $M_i+M_j$ for $j<i$. 
If $M_i$ and $M_j$ have the same pivot before the column addition,
the pivot of $M_i$ decreases under the column addition (or the column $M_i$ becomes zero, if $M_i=M_j$).

\define{Matrix reduction} is the process of repeatedly performing left-to-right column additions
until no two columns have the same pivot. For concreteness, we fix the following version:
we traverse the columns from $1$ to $\column$ in order. 
At column $i$, as long as it is non-zero and has a pivot that appears as a pivot in some column $j<i$, we add column $j$ to column $i$.
The resulting matrix is called \define{reduced}.
\begin{algorithm}
    \caption{Matrix reduction} 
	\DontPrintSemicolon
	\KwIn{Boundary matrix $M = ( M_{1}, \dots M_c)$}
	\KwOut{Reduced matrix $M' = ( M'_{1}, \dots M'_c)$}
	\For{$i = 1,\dots,c$}{ 
	    $M'_{i} = M_{i}$ \\
            \While{\label{al_line_while} $M'_i\neq 0$ and there exists $j < i$ with $\lowp{M'_{j}}=\lowp{M'_{i}}$}{
				$M'_{i} \leftarrow M'_{i} + M'_{j}$}
		 
	}
\label{Al_standard}
\end{algorithm}

We define the \define{cost of a column addition} of the form $M_i\leftarrow M_i+M_j$
as $\fillin M_j$, i.e., the number of non-zero entries in the column that is added to $M_i$.
The \define{cost of a matrix reduction} for a matrix $M$ is then the added cost of all column
additions performed during the reduction, and it is denoted by $cost(M)$. 
The \define{fill-in} of a reduced matrix $M'$ is $\fillin M'$, 
the number of non-zero entries of $M'$.
We can relate the cost of reducing a matrix $M$
to the fill-in of the reduced matrix as follows.
\begin{lemma}\label{lem:cost_bound}
For a matrix $M$ with $\column$ columns, let $M'$ denote its reduced matrix. 
Then
\[cost(M)\leq\column\cdot \fillin M'\]
\end{lemma}
\begin{proof}
Let $M'_{\leq i}$ denote the matrix formed by the first $i$ columns of $M'$. 
Then, after the matrix reduction algorithm has traversed the first $i$ columns,
the partially reduced matrix agrees with $M'$ on the first $i$ columns. In order to reduce
column $i+1$, the algorithm adds some subset of columns of $M'_{\leq i}$ to $M_{i+1}$,
each column at most once. Hence, the cost of reducing column $M_{i+1}$ is bounded by $\fillin M'_{\leq i}$.
We can therefore bound
\[cost(M)\leq \sum_{i=1}^{\column}\fillin M'_{\leq i}\leq \sum_{i=1}^\column \fillin M' = \column\fillin M'.\qedhere\]
\end{proof}

We interpret the cost of $M$ as a model of the (bit) complexity for performing matrix reduction.
Indeed, in practice, we will apply matrix reduction on (initially) sparse matrices whose columns are usually represented to contain only the indices of their non-zero entries to reduce memory consumption. 
If we arrange these indices
in a balanced binary search tree structure, for instance, performing a column operation $M_i\leftarrow M_i+M_j$
can be realized in $\mathcal{O}(\fillin M_j\log\fillin M_i)$ time, which matches our cost up to a logarithmic factor.
Alternatively, we can store columns as linked lists of non-zero indices and then, reducing column $i$, 
we can transform the column in a $\{0,1\}$-vector of length $\row$, and perform all additions in time proportional to $\fillin M_j$, resulting in a total complexity of $\mathcal{O}(\column\cdot(\row+\fillin M'))$. 
This complexity matches
$cost(M)$ if the reduced matrix has at least $\Omega(\row)$ non-zero entries (which will be the case for the cases studied
in this paper). 
We refer to~\cite{bauer2017phat} for
a more thorough discussion on the possible choices of data structures for (sparse) matrices.

\paragraph{Constant indices and pivotal indices.}
We call a column $M_i$ \define{constant}
  if it is not modified during matrix reduction, that is, $M_i=M'_i$. In this case, the row index of the pivot of its reduction $\lowp{M'_i}$ is a \define{constant index}. Otherwise, if $M_i\neq M'_i$ and the reduced column $M'_i$ is not zero, we will refer to $M'_i$ as a \define{pivotal} column. The row index of its pivot $\lowp{M'_i}$ is then a \define{pivotal index}. See Figure~\ref{fig:mat_stairs} for an illustration of these concepts.
\begin{figure}[H]
    \centering
    \begin{blockarray}{ccccc}
    & abc & acd & abd & bcd \\
    \begin{block}{c[cccc]}
    bc & 1 &   &  & 1 \\
    ad &   & 1 & 1 &   \\
    ab & 1 &   & 1 &   \\
    cd &   & 1 &   & 1 \\
    ac & \DoTikzmark{stepind1}{1} & 1 &   &   \\
    bd &   &   & \DoTikzmark{stepind2}{1} & 1 \\
    \end{block}
    \end{blockarray}
    \colrow[red, opacity=.6]{stepind1}{stepind1}
    \colrow[red, opacity=.6]{stepind2}{stepind2}
    \hspace{-0.4cm}
    $\xrightarrow[\text{reduction}]{\text{left-to-right column}}$
    \begin{blockarray}{ccccc}
    & abc & acd & abd & bcd \\
    \begin{block}{c[cccc]} 
        bc & 1 & 1 &   &  \\
        ad &   & 1 & 1 &  \\
        ab & 1 & 1 & 1 &  \\
        cd &   & \DoTikzmark{critind}{1} &   &  \\
        ac & \DoTikzmark{stepind1}{1} &  &   &  \\
        bd &   &   & \DoTikzmark{stepind2}{1} &  \\
    \end{block}
    \end{blockarray}
    \colrow[red, opacity=.6]{stepind1}{stepind1}
        \colrow[red, opacity=.6]{stepind2}{stepind2}
    \colrow[green, opacity=.3]{critind}{critind}
    \caption{Example of a (1-boundary) matrix in (on the left) and its reduced matrix (on the right), from the complex in Figure \ref{fig:closing_edge}. The columns corresponding to the simplices $abc$ and $abd$ are constant columns, the column corresponding to $acd$ is pivotal. The row indices of the red elements are constant, and the row index of the green element is a pivotal index. The zeros are not displayed.}
    \label{fig:mat_stairs}
\end{figure}

The next lemma is simple, yet crucial for our approach:
\begin{lemma}\label{lem:fillup_staircase}
\[
\fillin M' \leq \sum_{\text{S is a constant column of $M$}} \fillin S + \sum_{\text{$p$ is a pivotal index}} p
\]
\end{lemma}
\begin{proof}
Each column of $M'$ is zero or has some pivot $p$, which is either a constant index or a pivotal index. 
If $p$ is a constant index, then it comes with a constant column of $M$, which remains unchanged. Otherwise, the number of non-zero entries of a pivotal column is bounded by its pivot.
\end{proof}

\subsection{Simplicial filtrations, boundary matrices and homology}\label{subsec: Simplicial filtrations, boundary matrices and homology}
\paragraph{Simplicial complexes.}
Given a finite set $V$, a \define{simplicial complex} $K$ over $V$
is a collection of subsets of $V$, called \define{simplices}, with the property that if
$\sigma\in K$ and $\tau\subseteq \sigma$, also $\tau\in K$. A simplex with $(k+1)$-elements
is called \define{$k$-simplex}. 
$0$-, $1$-, and $2$-simplices are also called \define{vertices}, \define{edges},
and \define{triangles}, respectively. 
For a $k$-simplex $\sigma\in K$, any simplex $\tau$ with $\tau\subseteq \sigma$ is a \define{face} of $\sigma$ and a we call $\tau$ a \define{facet} of $\sigma$ whenever $\tau$ is $(k-1)$-dimensional.
The set of facets is called the \define{boundary} of $\sigma$.
A \define{subcomplex} $L$ of a simplicial complex $K$ is a subset of $K$ which is itself a simplicial complex.

\paragraph{Filtations.}%
A \define{filtered simplicial complex} is a simplicial complex $K$ with a fixed (total) ordering in every dimension. For the \v{C}ech, Vietoris-Rips and Erd\H{o}s--R\'enyi filtrations considered here, each simplex $\sigma\in K$ comes with a \define{entrance time} $\radius_\sigma\in\mathbb{R}_{\geq 0}$ such that for each face $\tau$ of $\sigma$, $\radius_\tau\leq \radius_\sigma$. 
Putting together all simplices in $K$ with the same entrance time $\radius\geq 0$ yields a subcomplex $K_{\radius}$ of $K$ and ordering all entrance times in an ascending way leads to a nested sequence of subcomplexes $\emptyset\subseteq K_{\radius_0}\subseteq K_{\radius_1} \subseteq \cdots \subseteq K_{\radius_N}=K$ which we call a \define{filtration} of $K$. Given any dimension $k$, we fix a total order on the set of $k$-simplices by first ordering by ascending entrance times, and then ordering simplices with equal entrance time either arbitrarily, or by some explicit rule specified later. 

\paragraph{Boundary matrices.}%
Let $K$ be a filtered simplicial complex with $\row$ $k$-simplices and $\column$ $(k+1)$-simplices. The $k$-dimensional \define{boundary matrix} of $K$ is a $(\row\times\column)$-matrix, where the order of the rows and columns is induced by the order on the underlying filtered simplicial complex and the entry $(i,j)$ in the matrix
is $1$ if the $i$-th $k$-simplex is a facet of the $j$-th $(k+1)$-simplex of $K$, and $0$ otherwise. 
We interpret boundary matrices as matrices over $\mathbb{Z}_2$. 
For a boundary matrix $\bdmat$ in dimension $k$, we have that $\fillin\bdmat=(k+2)\cdot\column$ because
every $(k+1)$-simplex has $k+2$ facets. 
Hence, boundary matrices are sparse, but this sparsity is not necessarily preserved by matrix reduction~\cite{morozov2005persistence}.

\paragraph{Homology.}%
We recall the basic notions of simplicial homology (with coefficients over the field $\mathbb{Z}_2$):
for a simplicial complex $K$, the \define{$k$-th chain group} $C_k$ is the vector space over $\mathbb{Z}_2$ that has the $k$-simplices
of $K$ as basis elements. 
Let $\partial_k:C_k\to C_{k-1}$ denote the unique homomorphism that maps every $k$-simplex $\sigma$
to the sum of its facets. We call the kernel $Z_k$ of $C_k$ the \define{$k$-th cycle group} and the image $B_k$ of $\partial_{k+1}$
the \define{$k$-th boundary group}; note $B_k\subseteq Z_k$ because $\partial_k\circ\partial_{k+1}=0$. 
The \define{$k$-th homology group} $H_k$ of $K$ is then defined as $Z_k/B_k$. 
Note that despite the name ``group'' for chains, cycles, boundaries,
and homologies, all these objects are vector spaces (because we take coefficients over $\mathbb{Z}_2$).
The \define{$k$-th Betti number} of $K$, denoted by $\betti_k(K)$, is the dimension of $H_k(K)$.

\paragraph{Persistent homology and matrix reductions.}
 Matrix reduction gives a wealth of information when applied to a filtered boundary matrix $\bdmat$.
For once, it yields the rank of $\bdmat$, which can be used, for instance, to compute the Betti numbers
of the simplicial complex: Writing $\bdmat$ and $\overline{\bdmat}$ for the boundary matrix in dimension $k$ and $k-1$, respectively, and $n_k$ for the
number of $k$-simplices, we have that $\betti_k(K)=n_k-\mathrm{rank}(\bdmat)-\mathrm{rank}(\overline{\bdmat})$. 
Moreover, because the matrix reduction respects the order
of the simplices, the pivots of the reduced matrix $D'$ yield the so-called \define{persistent barcode}
of the filtered simplicial complex consisting of the following two sets of intervals:
\begin{equation*}
\begin{split}
&P:=\{(\rho_i,\rho_j)\mid D'_j\neq 0 \text{ and } i=\lowp{D'_j} \} \\
&E:=\{(\rho_i,\infty)\mid D'_i=0 \text{ and } i\neq \lowp{D'_j} \; \forall j=1,\dots,n\}.
\end{split}
\end{equation*}
Here $\rho_i$ denotes the entrance time of the simplex associated to the $i$-th row or column of the boundary matrix (or its reduction). 
The intervals $(\rho_i,\rho_j)$ in the persistence barcode can be seen as topological features in the filtration appearing at entrance time $\rho_i$ and disappearing at $\rho_j$. 
The \define{persistence} $\rho_j-\rho_i$ of an interval is often interpreted as the significance of the observed topological feature. Consequently, intervals of the form $(\rho_i,\infty)$ can be seen as topological features which appear at scale $\rho_i$ and never disappear. 
We refer to~\cite{edelsbrunner2010computational,oudot2015persistence} for further details about
persistent homology and barcodes.

\paragraph{Good orders.}%
Let $K$ be a filtered simplicial complex. A pair $(\sigma,\tau)$ in $K$ is an \define{apparent pair} \cite{bauer2021ripser}, if $\sigma$ is the maximal facet of $\tau$ and $\tau$ is the minimal cofacet of $\sigma$ with respect to the total order on $K$. 
We say that $K$ is in \define{good order} if every interval in the persistence barcode of $K$ with zero persistence comes from an apparent pair. 
The good order property of a filtration is equivalent to the following criterion on its boundary matrix: 
No pivotal index is allowed to have the same entrance time as the column that contains it as a pivot. 
We will show that for Vietoris--Rips and \v{C}ech filtrations, we can always achieve a good order by lexicographically sorting rows and columns with the same entrance time (Lemma \ref{lem:concurrent_submatrices_good_orders}).  
In \cref{ex_good_order_crucial}, we show why having filtrations in good order is crucial for our results.

Not every filtration can be brought into good order, no matter how simplices with same entrance time are sorted. For an example, consider the triangulation of a dunce hat (see Figure \ref{fig:dunce_hat}) and its boundary matrix in dimension $1$.
Assume that all simplices have the same entrance time. Since every edge has $2$ or $3$ incident triangles, at least one column addition is performed, no matter what order is chosen. A zero column in the reduced matrix would imply that the $2$-nd Betti number is positive, but this contradicts the fact that the dunce hat is contractible

\begin{figure}[H]
    \centering
    \includegraphics[width=7cm]{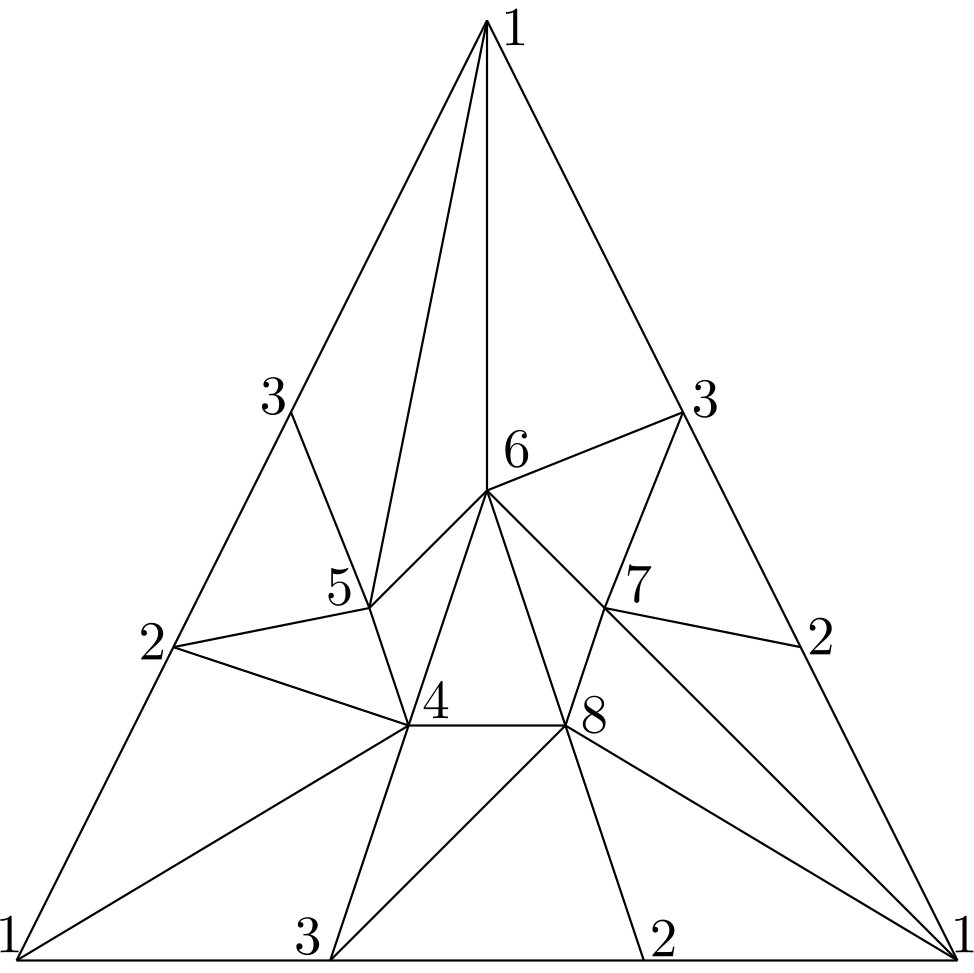}
    \caption{Triangulation of a dunce hat (see also Example $1.22$ in \cite{scoville2019discrete}).}
    \label{fig:dunce_hat}
\end{figure}

\section{Fill-in analysis}\label{Fill-in analysis}
From now on, we fix $K$ to be a complete filtered simplicial complex over $\ver$ vertices, yielding
a $k$-dimensional boundary matrix $\bdmat$ with $\row=\binom{\ver}{k+1}=\Theta(\ver^{k+1})$ rows and $\column=\binom{\ver}{k+2}=\Theta(\ver^{k+2})$ columns.  
Let $\bdmat'$ be the reduction of $\bdmat$. 
We have a simple lower bound on the fill-in of $D'$:
\begin{lemma}\label{lem:fillup_lower_bound}
If $D'$ is a reduction of a $k$-boundary matrix $D$, then
$\fillin\bdmat'\geq \binom{\ver}{k+1}-\binom{\ver}{k}=\Omega(n^{k+1})$.
\end{lemma}
\begin{proof}
Observe that $\fillin\bdmat'\geq\mathrm{rank}(\bdmat)$ because there
will be at least one non-zero entry at the pivot entries.
By the aforementioned formula of Betti numbers, we have
$\mathrm{rank}(\bdmat)=n_k-\betti_k(K)-\mathrm{rank}(\overline{\bdmat})$, where $\overline{\bdmat}$ denotes the boundary matrix in dimension $k-1$.
 The statement
follows because the Betti number of $K$ is $0$ in all dimensions $k\geq 1$ and $\overline{\bdmat}$ is a $(\binom{\ver}{k}\times\binom{\ver}{k+1})$-matrix whose rank is at most
$\binom{\ver}{k}$.
\end{proof}

We now turn to an upper bound for $\fillin\bdmat'$. By \cref{lem:fillup_staircase}, we get
\begin{equation}
\label{eqn:det_fillup}
\fillin \bdmat' \leq (k+2)\binom{n}{k+1} + \sum_{\text{$p$ is a pivotal index}} p
\end{equation}
because every column of $\bdmat$ has precisely $k+2$ non-zero entries and, since there are $\binom{n}{k+1}$ rows, i.e., possible pivots, there cannot be more constant columns than number of possible pivots.
We assume $k$ to be a constant,
in which case the first term further simplifies to $\Theta(n^{k+1})$. 
To bound the second term, two observations are necessary. 
First, it is possible to link the presence of a pivotal index $p$ with non-zero (k+1)-homology of $K_{\rho_p}$ (see Figure~\ref{fig:closing_edge}). 
Second, in all filtrations considered in what follows, non-zero homology of a complex in the filtration will turn out to be an unlikely event after some threshold radius, which makes the appearance of pivotal indices unlikely as well. 

\begin{lemma}\label{lem:critical_indices_nonzero_betti}
Let $D$ be the $k$-dimensional boundary matrix of a filtered simplicial complex in good order and let $D'$ be its reduction. If $(p,c_p)$ is a pivotal entry of $D'$, then $\beta_k(K_{\rho_p})>0$. 
\end{lemma}

\begin{proof}
Let $z$ be the cycle in $K_{\rho_p}$ which arises as a sum of the $k$-simplices associated to the non-zero entries of $D'_{c_p}$. After matrix reduction, the first $c_p$ columns of $D'$ are linearly independent, in particular, $D'_{c_p}$ can not be written as a linear combination of columns in $D$ with index smaller than $c_p$. This means that $z$ does not lie in the $k$-th boundary group of $K_{\rho_{c_p-1}}$. More precisely, since $D$ is in good order, $\rho_p<\rho_{c_p}$ and thus $z$ does not lie in the $k$-th boundary group of $K_{\rho_p}$. We conclude that $z$ is a non trivial cycle in $H_k(K_{\rho_p})$ and thus $\beta_k(K_{\rho_p})>0$.   
\end{proof}

\begin{figure}[H]
\centering
\begin{tikzpicture}
    \tikzstyle{point}=[circle,draw=black,fill=black,inner sep=0pt,minimum width=3pt,minimum height=3pt]
    \node (a0)[] at (-2.4,0.52) {};
    \node (a00)[] at (-1.3,0.52) {};
    \node (al0)[label={$\cdots$}] at (-2.6,0.2) {};
    \draw[right hook-latex] (a0) -- (a00);
    \node (a1)[point, label=above:{a}] at (-.9,1) {};
    \node (b1)[point, label=right:{c}] at (1.1,0) {};
    \node (c1)[point, label=above:{b}] at (0.9,1) {};
    \node (d1)[point, label=left:{d}] at (-1.1,0) {};
    \draw  (a1.center) -- (c1.center);
    \draw  (c1.center) -- (b1.center);
    \draw  (a1.center) -- (d1.center);
    \node (b0)[] at (1.5,0.52) {};
    \node (b00)[] at (2.6,0.52) {};
    \draw[right hook-latex] (b0) -- (b00);
    \node (a2)[point, label=above:{a}] at (3.1,1) {};
    \node (b2)[point, label=right:{c}] at (5.1,0) {};
    \node (c2)[point, label=above:{b}] at (4.9,1) {};
    \node (d2)[point, label=left:{d}] at (2.9,0) {};
    \draw  (a2.center) -- (c2.center);
    \draw  (c2.center) -- (b2.center);
    \draw  (a2.center) -- (d2.center);
    \draw  (b2.center) -- (d2.center);
    \node (c0)[] at (5.5,0.52) {};
    \node (c00)[] at (6.6,0.52) {};
    \draw[right hook-latex] (c0) -- (c00);
    \node (a3)[point, label=above:{a}] at (7.1,1) {};
    \node (b3)[point, label=right:{c}] at (9.1,0) {};
    \node (c3)[point, label=above:{b}] at (8.9,1) {};
    \node (d3)[point, label=left:{d}] at (6.9,0) {};
    \draw  (a3.center) -- (c3.center);
    \draw  (c3.center) -- (b3.center);
    \draw  (a3.center) -- (d3.center);
    \draw  (b3.center) -- (d3.center);
    \draw  (a3.center) -- (b3.center);
    \fill[gray, opacity=0.2] (a3.center) -- (b3.center) -- (d3.center) -- cycle; 
    \fill[gray, opacity=0.3] (a3.center) -- (b3.center) -- (c3.center) -- cycle; 
    \node (d0)[] at (9.5,0.52) {};
    \node (d00)[] at (10.6,0.52) {};
    \node (dl0)[label={$\cdots$}] at (10.8,0.2) {};
    \draw[right hook-latex] (d0) -- (d00);
\end{tikzpicture}
\caption{The insertion of the edge $cd$ - that has pivotal index as Figure \ref{fig:mat_stairs} shows - creates a $1$-cycle, i.e., increases $\betti_1$.}
\label{fig:closing_edge}
\end{figure}

In combination with (\ref{eqn:det_fillup}), the lemma gives
\begin{equation}
\label{eqn:betti_fillup}
\fillin \bdmat' \leq \Theta(n^{k+1}) + \sum_{i=0}^\row i\cdot\indicator_{(\betti_k(K_{\rho_i})>0)}(i),
\end{equation}
where $\indicator_{(\betti(K_{\rho_i})>0)}(-)$ is the indicator function taking the value $0$ for all $i=0,\dots,r$ such that the $k$-th homology
group of $K_{\rho_i}$ is trivial, and $1$ otherwise. 
In the worst-case, all these homology groups are non-trivial,
and the bound yields $\mathcal{O}(\row^2)$, which can also be derived directly as an upper bound
on $\fillin \bdmat'$. 
A better upper bound can be obtained in expectation if a sample with favorable properties is drawn from the set of filtered simplicial complexes in good order. 
We then have the following main lemma which bounds the expected value of the fill-in.

\begin{lemma}\label{lem:main_lemma}
Let $K$ be a random filtered simplicial complex in good order such that for given constants $T\geq 1$ and $A>0$,
\[
\mathbb{P}(\betti_k(K_{\rho_i})>0)<\frac{A}{\row}
\]
for all row indices $i>T$ (with $\row=\binom{\ver}{k+1}$ as before). Further, let $D$ be the boundary matrix of $K$ and $D'$ be its reduction. We then have that
\[
\mathbb{E}[\fillin\bdmat']=\mathcal{O}((1+A)n^{k+1}+T^2).
\]
\end{lemma}
\begin{proof}
We split the second term of (\ref{eqn:betti_fillup}) in two parts:
\begin{equation*}
\mathbb{E}\left[\sum_{i=0}^\row i\cdot\indicator_{(\betti_k(K_{\rho_i})>0)}\right]=\sum_{i=0}^{T-1} i\cdot\underbrace{\mathbb{P}(\betti_k(K_{\rho_i})>0)}_{\leq 1}+\sum_{i=T}^{\row} i\cdot\underbrace{\mathbb{P}(\betti_k(K_{\rho_i})>0)}_{\leq \frac{A}{\row}}\leq T^2 + \sum_{i=T}^{\row} A \leq T^2 + A\row \, .
\end{equation*}
The result follows by applying the latter inequality to the second term in the right-hand side of the inequality in (\ref{eqn:betti_fillup}) and taking the expectation.
\end{proof}

Before moving on to the next sections, where we will see that the assumption in this lemma holds for three examples of filtered simplicial complexes, we provide a simple example showing why the hypothesis of good order is crucial in our approach. 

\begin{example}\label{ex_good_order_crucial}
Consider the filtered simplicial complex depicted in \cref{fig:non_good_order}, with entrance times $s_0 < s_1 < s_2 < s_3 < t$.
\begin{figure}[H]
	\centering\tikzset{every picture/.style={line width=0.75pt}} 
	\begin{tikzpicture}  
		\tikzstyle{point}=[circle,draw=black,fill=black,inner sep=0pt,minimum width=3pt,minimum height=3pt]
		\node (a3)[point, label=above:{\textbf{a}}, label=left:{$s_0$}] at (7.1,1) {};
		\node (b3)[point, label=right:{\textbf{c}}, label=below:{$s_0$}] at (9.1,0) {};
		\node (c3)[point, label=above:{\textbf{b}}, label=right:{$s_0$}] at (8.4,1.1) {};
		\node (d3)[point, label=left:{\textbf{d}}, label=below:{$s_0$}] at (7.6,-0.6) {};
		\draw  (a3.center) -- (c3.center);
		\draw  (c3.center) -- (b3.center);
		\draw  (a3.center) -- (d3.center);
		\draw  (b3.center) -- (d3.center);
		\draw  (a3.center) -- (b3.center);
		\fill[gray, opacity=0.2] (a3.center) -- (b3.center) -- (d3.center) -- cycle; 
		\fill[gray, opacity=0.3] (a3.center) -- (b3.center) -- (c3.center) -- cycle; 
		\node (l0)[] at (9.05,0.7) {$s_2$};
		\node (l1)[] at (7.8,1.25) {$s_1$};
		\node (l3)[] at (8.45,-.6) {$s_3$};
		\node (l4)[] at (7.15,0.15) {$t$};
		\node (l5)[] at (8,0) {$t$};
		\node (l6)[] at (8.3,.7) {$t$};
		\node (l6)[] at (7.9,.8) {$t$};
		\node (K)[] at (5.5,0.52) {$K=$};
	\end{tikzpicture}
	\caption{A filtered simplicial complex $K$ over the vertices $\{a,b,c,d\}$, with the entrance times $s_0 < s_1 < s_2 < s_3 < t$ depicted near the corresponding simplices.}
	\label{fig:non_good_order}
\end{figure}

Choosing $ad<ac$, $acd<abc$, we obtain the following $1$-dimensional boundary and reduced matrices:

\begin{figure}[H]
	\centering
	\begin{blockarray}{ccc}
		& acd & abc  \\
		\begin{block}{c[cc]}
			ab &  &  1  \\
			bc &   & 1   \\
			cd & 1 &    \\
			ad & 1  &  \\
			ac & 1 & 1  \\
		\end{block}
	\end{blockarray}
	$\xrightarrow[\text{reduction}]{\text{left-to-right column}}$
	\begin{blockarray}{ccc}
		& acd & abc  \\
		\begin{block}{c[cc]}
			ab &  &  1  \\
			bc &   & 1   \\
			cd & 1 & 1   \\
			ad & 1 & \DoTikzmark{critind}{1} \\
			ac & \DoTikzmark{stepind}{1} &   \\
		\end{block}
	\end{blockarray}
	\colrow[red, opacity=.3]{stepind}{stepind}
	\colrow[green, opacity=.3]{critind}{critind}
	
	\caption{The $1$-dim boundary matrix of the filtration from Figure \ref{fig:non_good_order} and its reduction.}
	\label{fig:non_good_boundary}
\end{figure}

In this case, the pivotal index $ad$ has the same entrance time $t$ as its column $abd$, thus the order is not \textit{good}. 
The $1$-dimensional Betti number of $K_t$ is zero, which is problematic for our model as pivotal entries cannot be related here to non-zero homology. 
This in turn might lead to fill-in which cannot be controlled by homology. 
Note that the filtered simplicial complex in Figure \ref{fig:non_good_order} can be easily brought into good order by setting $ac<ad$ and $abc<acd$. 
In other words, if these simplices are ordered lexicographically according to the vertices's order, the filtration is in good order.
\end{example}

\section{Complexity for the \v{C}ech filtration}\label{section_cech}
We begin this section with the description of random \v{C}ech filtrations and show that these are in good order. 
Then we present a crucial result by Kahle \cite{kahle2011random} together with a probabilistic lemma, which shows that our random model suits the prerequisites of Lemma \ref{lem:main_lemma}. 
Finally, we state and prove the first main theorem of this paper, providing bounds on average fill-in and cost of matrix reduction of random \v{C}ech filtrations. 

\paragraph{\v{C}ech filtrations}
Given $n$-points $\mathcal{P}_n=\{p_1,\dots,p_n\}$, the \define{\v{C}ech complex at scale $\alpha$}, $\check{C}(\mathcal{P}_n,\alpha)$, is the simplicial complex with vertices $\mathcal{P}_n$ and all simplices $\sigma$ such that $\bigcap_{p_i\in\sigma}B_{\alpha}(p_i)\neq \emptyset$. 
The entrance time $\rho_\sigma$ (see Section \ref{subsec: Simplicial filtrations, boundary matrices and homology}) of a simplex $\sigma$ is then the smallest $\alpha$ such that $\sigma\in \check{C}(\mathcal{P}_n,\alpha)$. 
Equivalently, a $k$-simplex $\sigma=\left\{p_{i_0},\dots,p_{i_k} \right\}$ is in $\check{C}(\mathcal{P}_n,\alpha)$ whenever the smallest ball which contains the points of $\sigma$ has radius $\alpha$. 
We call this ball the \define{minimal enclosing ball of $\sigma$}, in short $\text{MEB}(\sigma)$. 
The $\text{MEB}$ of every simplex is unique \cite[Lemma 4.14]{de2013computational}. 
If this were not the case, then the ball centered in the midpoint $z$ of the two $\text{MEB}$ centers, having radius $\lVert z - p \rVert$ with $p$ an arbitrary point in the intersection of the boundaries, would contain $\sigma$. 
Moreover, this ball would have radius smaller than the MEB, a contradiction. 
The proof of the following lemma goes along the same lines. 

\begin{lemma}
\label{lem:MEB_same}
Let $\sigma\subseteq \tau \subseteq \mathbb{R}^d$ be simplices of a \v{C}ech complex having the same $\text{MEB}$-radius. Then $\text{MEB}(\sigma)=\text{MEB}(\tau)$. 
\end{lemma}

Letting the radius $\alpha$ of the balls range from zero, where $\check{C}(\mathcal{P}_n,0)$ equals $\mathcal{P}_n$, to the smallest value $\alpha_M$ such that $\check{C}(\mathcal{P}_n,\alpha_M)$ is the full simplicial complex,   yields a filtration of simplicial complexes, the \define{\v{C}ech filtration} $\check{C}(\mathcal{P}_n)$. 

\paragraph{Random \v{C}ech filtrations} 
We pick $n$-points $\mathcal{X}_n$ uniformly i.i.d. in the unit cube ${[-\frac{1}{2},\frac{1}{2}]}^d$ and build the \v{C}ech filtration $\check{C}(\mathcal{X}_n)$ on these points, ordering the simplices
\begin{enumerate}[label=(\roman{enumi})]
\item by their entrance time, 
\item then by their dimension,
\item then by their lexicographic order induced by the total order on the vertices $\mathcal{X}_n$ (which is given by their indices).
\end{enumerate}
This order is called a $\rho$-lexicographic order $\leq_\rho$. 
The introduced filtration of simplicial complexes $\check{C}(\mathcal{X}_n)$ is the \define{\v{C}ech filtration model}. 
The $\rho$-lexicographic order is indeed a good order on $\check{C}(\mathcal{X}_n)$ as we will see in the next paragraph.
Notice that using e.g. radix sort, a lexicographic order on the $k$-simplices can be obtained in $\mathcal{O}(k \binom{\ver}{k+1})$ time. Similarly, the $(k+1)$-simplices can be ordered in $\mathcal{O}((k+1)\binom{\ver}{k+2})=\mathcal{O}(n^{k+2})$ time. Thus, the time complexity of ordering a $k$-dimensional boundary matrix is $\mathcal{O}(n^{k+2})$ and therefore negligible in view of the overall complexity of matrix reduction.

\paragraph{Proof of good order.}
We assume for simplicity that the point set $\mathcal{X}_n$ is in generic position, meaning that whenever $\sigma$ and $\tau$ have different $\text{MEB}s$, their $\text{MEB}s$ have different radii. 
In particular, this implies that all edges have different lengths. 
This property is indeed generic since it is ensured by a random perturbation of the points, and hence the probabilty that a non-generic $\mathcal{X}_n$ is sampled is $0$.

If we restrict our view to the vertices, edges and triangles of $\check{C}(\mathcal{X}_n)$, it is not necessary to sort $\rho$-lexicographically to achieve a good order. 
Since each edge has a different entrance time, the order on edges is canonical. 
It then suffices to resolve tie-breaks on triangles with the same entrance time arbitrarily. 
Indeed any pivotal column of the boundary matrix undergoes at least one reduction step. 
Since every edge has a different entrance time, this means that the entrance time of a pivotal index is always strictly smaller than the entrance time of its column. 

To prove that the $\rho$-lexicographic order on $\check{C}(\mathcal{X}_n)$ is a good order, we then consider boundary matrices in arbitrary dimension. 
It suffices to limit our view to concurrent submatrices. 
A \define{concurrent submatrix} $C\in\mathbb{Z}_2^{\row'\times\column'}$ of a $k$-dimensional boundary matrix $D\in\mathbb{Z}_2^{\row\times\column}$ of $\check{C}(\mathcal{X}_n)$ is the restriction of $D$ to a set of successive rows and columns which have the same entrance time. 
If matrix reduction yields no pivotal index in any concurrent submatrix of any boundary matrix, then the whole filtration is in good order. 
As the following lemma shows, this statement is fulfilled and hence $\leq_\rho$ is a good order on $\check{C}(\mathcal{X}_n)$.

\begin{lemma}
\label{lem:concurrent_submatrices_good_orders}
Let $C\in\mathbb{Z}_2^{r\times c}$ be a $k$-dimensional concurrent submatrix of a \v{C}ech filtration. Then matrix reduction either leaves a column unchanged or it reduces that column to zero. 
\end{lemma}

In our case, the function which maps each simplex of a \v{C}ech filtration to its $\text{MEB}$-radius is a generalized discrete Morse function \cite{bauer_morse_cech_delaunay}. Therefore, the proof of Lemma \ref{lem:concurrent_submatrices_good_orders} immediately follows from Lemma $9$ in \cite{br-apparent}. The fact that $\rho$-lexicographic orders are good orders can be stated more generally using discrete Morse theory, see \cite{br-apparent}. We nevertheless present an elementary proof for the purpose of self-containment.

\begin{proof}[Proof of Lemma \ref{lem:concurrent_submatrices_good_orders}]
By our genericity assumption, every simplex associated to a row or column in $C$ contains the same $l$-dimensional face $\xi$ with $l<k$. We start by removing the vertices of $\xi$ from these simplices. 
This operation preserves the order on rows and columns and thus keeps $C$ unchanged.  

Assume that two columns $\nu\leq_\rho\tau$ get added during matrix reduction, that is, they have the same pivot $\sigma$. 
Let $x$ and $y$ be the single elements contained in $\nu\setminus\sigma$ and $\tau\setminus\sigma$, respectively. 
Then $x$ is smaller than $y$ in the $\rho$-lexicographical order since $\nu$ and $\tau$ have $\sigma$ as pivot. 
Further, the smallest vertex in $\nu$ has to be $x$ and the smallest element in $\tau$ has to be $y$ because otherwise $\sigma$ could not be the common pivot of these simplices.  

The $k-l+1$ non-zero entries in column $\tau$ are indexed by the $(k-l)$-element subsets $\sigma_1 < \cdots <\sigma_{k-l+1}=\sigma$ of $\tau$. Now both $\nu\cup \xi$ and $\tau \cup \xi$ share the same minimal enclosing ball by Lemma \ref{lem:MEB_same}. 
Therefore, all $(k-l+1)$ element subsets of $\nu \cup \tau$ correspond to columns in $C$. 
More precisely, there are columns in $C$ indexed by $x$ concatenated by $\sigma_i$ and we write such a column as $x*\sigma_i$. The pivot of these columns is $\sigma_i$ because $\sigma_i\subseteq \tau$ and $x$ is strictly smaller than $y$, the smallest element in $\tau$. 
Therefore, adding the columns $x*\sigma_1<\cdots <x*\sigma_{k-l+1}$ to $\tau$ reduces the non-zero entries at $\sigma_1<\cdots < \sigma_{k-l+1}$. 
It remains to show that after these reductions, no non-zero entry appears in $\tau$ in a row smaller than $\sigma_1$. 
The rows in question are precisely indexed by the concatenations of $x$ with the $(k-l-1)$-element subsets $\eta_i$ of $\tau$. 
Since $\lvert(\nu\cup \tau )\backslash (x\cup\eta_i) \rvert=2$ there are exactly two columns in $\{x*
\sigma_1,\cdots,x*\sigma_{k-l+1}\}$ which have a non-zero entry in row $x*\eta_i$. Hence it can be concluded that the column indexed by $\tau$ gets reduced to zero. 
\end{proof}

\paragraph{Probabilistic results.}

Recall from Lemma \ref{lem:critical_indices_nonzero_betti} that pivotal indices in a good order witness non-trivial homology groups. 
However, non-trivial homology becomes unlikely at larger scales. 
The following result is due to Kahle \cite[Theorem 6.1]{kahle2011random}; 
we restate it taking into account the speed of convergence.
See Appendix~\ref{A_lemma_C} for a proof.

\begin{lemma}
\label{lem:CentralLemmaB_CechCase}
Given positive integers $\ell$ and $k$, there exists a constant $c>0$ such that for  $\alpha \geq c {\left(\frac{\log(n)}{n}\right)}^{1/d}$:
\begin{equation}
\label{eq:CentralEqB_CechCase}
\mathbb{P} \left( \beta_k(\check{C}(\mathcal{X}_n,\alpha)) \neq 0 \right) \leq \frac{1}{n^\ell}
\end{equation}
for sufficiently large $n$. 
\end{lemma}

We set $\alpha^*\coloneqq c{\left(\frac{\log n}{n}\right)}^{1/d}$ and claim that for sufficiently large row indices $i$, it is unlikely that the entrance time $\rho_i$ is smaller than $\alpha^*$. This statement is a direct consequence of the next technical lemma. 

\begin{lemma}
\label{lem:SecondLemmaB_CechCase}
Let $N_k=N_k(\mathcal{X}_n)$ be the number of $k$-simplices in $\check{C}(\mathcal{X}_n,\alpha^*)$. 
Then for any $\ell\in\mathbb{N}$ there exists a constant $\lambda >0$ such that we have
\begin{equation}
	\label{eq:conc_ineq}
	\mathbb{P}\left(N_k\geq \lambda n \log^k n\right) \leq \frac{1}{n^\ell},
\end{equation}
for sufficiently large $n$. 
\end{lemma}

We postpone the proof of the two lemmas above to Appendix \ref{A_lemma_C}. All results needed to prove the main theorem of this section are at hand now. 

\begin{mthm}
\label{thm:Main_thm_Cech}
Let $D'$ be the reduced $k$-dimensional boundary matrix of the \v{C}ech filtration $\cechfilt$. Then 
\begin{equation*}
\mathbb{E}[\# D']=\mathcal{O}(n^{2}\log(n)^{2k}+n^{k+1})
\end{equation*}
and the cost of matrix reduction is bounded by
$\mathcal{O}(n^{k+4}\log(n)^{2k}+ n^{2k+3})$.
\end{mthm}

\begin{proof}
By \cref{lem:CentralLemmaB_CechCase} there exists a constant $c>0$ such that $\mathbb{P} \left( \beta_k(\check{C}(\mathcal{X}_n,\alpha)) \neq 0 \right) \leq \frac{1}{n^{k+1}}$ for $\alpha\geq c {\left(\frac{\log n}{n}\right)}^{1/d}\coloneqq\alpha^*$.
Depending on $\alpha^*$, there exists a $\lambda>0$ such that $\mathbb{P}(N_k\geq \lambda n \log^k n)\leq \frac{1}{n^{k+1}}$. 
We set $T\coloneqq\lambda n \log^k n$. 
For every row index $i>T$ we have that $\rho_i\leq \alpha^*$ implies that there must be at least $T$ $k$-simplices with entrance time smaller than $\alpha^*$, i.e. $\mathbb{P}(\rho_i\leq \alpha^*)\leq \mathbb{P}(N_t\geq T)\leq \frac{1}{n^{k+1}}$. Putting everything together yields
\begin{equation*}
\begin{split}
\Prob{\beta_k(\check{C}(\mathcal{X}_n,\rho_i)> 0} & =\Prob{\beta_k(\check{C}(\mathcal{X}_n,\rho_i)) > 0\wedge \rho_i>\alpha^*}+\Prob{\beta_k(\check{C}(\mathcal{X}_n,\rho_i)) > 0\wedge \rho_i\leq\alpha^*} 
\\
&\leq \Prob{ \beta_k(\check{C}(\mathcal{X}_n,\rho_i))\neq 0\mid \rho_i>\alpha^*} + \Prob{\rho_i\leq\alpha^*}
\\
&\leq \frac{1}{n^{k+1}}+\frac{1}{n^{k+1}}.
\end{split}
\end{equation*}

We have shown that the conditions of \cref{lem:main_lemma} are met, and thus the statement about the expected fill-in follows from that result. 
With the expected fill in at hand, the cost of matrix reduction is a consequence of \cref{lem:cost_bound} as the number of columns is $\Theta(n^{k+2})$. 
\end{proof}

In dimension $1$, our bound simplifies to $\mathbb{E}[\# D']=\mathcal{O}(n^{2}\log^{2}n)$ and implies that the reduced matrix has fewer entries in expectation than the unreduced boundary matrix which has precisely $3\binom{n}{3}=\Theta(n^3)$ non-zero entries. 
Moreover, since \cref{lem:fillup_lower_bound} implies that
the expected fill-in cannot be smaller than quadratic in $n$, our bound is tight up to a factor of $\log^2n$. 
Even stronger is the case $k>1$: the second summand is asymptotically dominant, i.e. we obtain $\mathcal{O}(n^{k+1})=\mathbb{E}[\# D']\geq \Omega(n^{k+1})$ by \cref{lem:fillup_lower_bound}. 
Therefore, our fill-in bound is tight for $k>1$. 

\section{Complexity for the Vietoris--Rips filtration}\label{section_rips}
The organization of this section is analogous to Section \ref{section_cech}. 
As clique filtrations, Vietoris--Rips (in short VR) filtrations are more combinatorial in nature than \v{C}ech filtrations, but their boundary matrices are very similar regarding our scope. We will thus point out differences to Section \ref{section_cech} rather then justify each step again. 

\paragraph{Vietoris--Rips filtrations.} 
Given $n$ points $\mathcal{P}_n=\{p_1,\dots,p_n\}\subseteq{[-\frac{1}{2},\frac{1}{2}]}^d$, the \define{Vietoris--Rips complex at scale $\alpha$}, $\text{VR}(\mathcal{P}_n,\alpha)$ is the simplicial complex with vertices $\mathcal{P}_n$ and all simplices $\sigma$ such that $B_{\alpha}(p_i)\cap B_{\alpha}(p_j)\neq \emptyset$ for all $p_i,p_j\in\sigma$, i.e. the diameter of $\sigma$ is at most $\alpha$. Letting the scope $\alpha$ range from zero, where $\text{VR}(\mathcal{P}_n,0)=\mathcal{P}_n$, to $\alpha_M$ such that $\text{VR}(\mathcal{P}_n,\alpha_M)$ is the full complex yields a nested sequence of simplicial complexes, the \define{Vietoris--Rips filtration} $\text{VR}(\mathcal{P}_n)$.

\paragraph{Clique filtrations.} 
Vietoris--Rips filtrations are a special case of so called clique filtrations. For a filtered simplicial complex $K$ with ordered $k$-simplices $\sigma_1,\ldots,\sigma_\row$, we define the \define{$k$-clique complex} $K_i$ with $0\leq i\leq\row$ as the largest subcomplex of $K$ that contains
exactly $\sigma_1,\ldots,\sigma_i$ as $k$-simplices. Note that each $K_i$ necessarily contains all $\ell$-simplices
of $K$ with $\ell<k$, but that is not the case for simplices of dimension $\geq k$. It holds that $K_i\subseteq K_j$ for $i<j$ and we thus have the \define{$k$-clique filtration} $K_0\subseteq K_1 \subseteq \cdots \subseteq K_r=K$. The filtration $\text{VR}(\mathcal{P}_n)$ is therefore a $1$-clique filtration of the full simplicial complex on $n$ vertices. For its construction, it suffices to have all pairwise distances between the points in $\mathcal{P}_n$. 

Boundary matrices $D$ of clique filtrations generically appear in staircase shape. That is, for any $j<i$ such that $D_i$ and $D_j$ are non-zero, the pivot of $D_j$ is not larger than the pivot of $D_i$. In this case the constant columns are easily identifiable as the ones with smallest index given their pivot. These column-pivot pairs are apparent pairs. 

\paragraph{Random Vietoris--Rips filtrations.}
As in the \v{C}ech case, $n$-points $\mathcal{X}_n=\{X_1,\dots,X_n\}$ are picked uniformly i.i.d. in the unit cube $[-\frac{1}{2},\frac{1}{2}]^d$. We now sort the vertices in $\mathcal{X}_n$ by ascending distance to the origin. 
Building the Vietoris--Rips filtration on $\mathcal{X}_n$ such that the simplices are ordered in $\rho$-lexicographic order yields the \define{Vietoris--Rips filtration model} $\VRfilt$. This is again a random filtered simplicial complex in good order.

\paragraph{Proof of good order.} As in Section \ref{section_cech} all edges in $\mathcal{X}_n$ are assumed to have different length. 
Thus, each edge gets added to $\VRfilt$ at a different entrance time, and each triangle gets added to $\VRfilt$ together with its longest edge. Ordering the edges in the natural way and resolving tie-breaks between triangles arbitrarily yields a staircase shaped $1$-dimensional boundary matrix $D$. 
In this case it is easy to see that the restriction of $\VRfilt$ to simplices of dimension $\leq 2$ is in good order. 

That $\VRfilt$ is in good order also in higher degrees can be seen by similar arguments as in the section on \v{C}ech filtrations. 
For the proof of the analogon of Lemma \ref{lem:concurrent_submatrices_good_orders}, note the following two properties of Vietoris--Rips filtrations. 
First, each $k+1$ simplex with entrance time $\rho$ has to contain the edge with the same entrance time. 
This holds due to the genericity assumption on edges and because each higher dimensional simplex enters at the same time as its latest edge. 
For the second property, assume that all $(k+1)$-simplices with entrance time $\rho$ share an $l$-dimensional face $\xi$ and two $(k+1)$-simplices $\tau$ and $\nu$ with entrance time $\rho$ share a $\leq_\rho$-maximal facet $\sigma$. 
Then the distance between the vertices $\tau\backslash \sigma$ and $\nu\backslash \sigma$ is smaller than $\rho$ since the vertices were ordered by distance to the origin. 
This means that the $(k+2)$-simplex $\tau\cup \nu$ has also entrance time $\rho$ and, in particular, that all $(k-l+1)$ element subsets of $\tau \cup \nu$ union $\xi$ have also entrance time $\rho$. \qed

As in the \v{C}ech case, the high-level reason why the $\rho$-lexicographic order is a good order is that, under the given genericity assumption on edges, the map assigning each simplex its entrance time is a generalized discrete Morse function \cite{br-apparent}.

\paragraph{Probabilistic results.} 

Again, a theorem of Kahle, \cite[Theorem 6.5]{kahle2011random} comes in handy. We restate it in order to include the speed of convergence. The proof is included in Appendix \ref{A_lemma_VR}.   
\begin{lemma}
\label{lemma_rips_beta}
Given positive integers $\ell$ and $k$, there exists a constant $c>0$ such that for  $\alpha \geq c {\left(\frac{\log(n)}{n}\right)}^{1/d}$:
\begin{equation}
\label{eq:CentralEqB_RipsCase}
\mathbb{P} \left( \beta_k(\text{VR}(\mathcal{X}_n,\alpha)) \neq 0 \right) \leq \frac{1}{n^\ell}
\end{equation}
for sufficiently large $n$. 
\end{lemma}

There is no need to prove a Vietoris--Rips version of Lemma \ref{lem:SecondLemmaB_CechCase}, as each \v{C}ech complex $\cechmod$ is a subcomplex of $\VRcompl$. We have the same main theorem as in the \v{C}ech case.

\begin{mthm}\label{mainthm_VR}
Let $D'$ be the reduced $k$-dimensional boundary matrix of the Vietoris--Rips filtration $\VRfilt$. Then 
\begin{equation*}
\mathbb{E}[\# D']=\mathcal{O}(n^{2}\log(n)^{2k}+n^{k+1})
\end{equation*} and the cost of matrix reduction is bounded by $\mathcal{O}(n^{k+4}\log(n)^{2k}+ n^{2k+3})$.
\end{mthm}

\begin{proof}
Similar to Theorem \ref{thm:Main_thm_Cech}. 
\end{proof}

\section{Complexity for the  Erd\H{o}s--R\'enyi  filtration}\label{section_erdos_renyi}
The last considered model does not arise from an underlying random point cloud and is therefore purely combinatorial. It shares the property of being a clique filtration with the Vietoris--Rips filtration and so this section will be structured similarly to the last one. 

\paragraph{Random Erd\H{o}s--R\'enyi filtrations.} Fix $n$ possible points and pick a real value in $[0,1]$ uniformly at random for each possible edge. The \define{random Erd\H{o}s--R\'enyi complex at scale $\alpha$}, $\text{ER}(n,\alpha)$ is the largest subcomplex of the full simplicial complex on $n$ vertices which contains all edges of length at most $\alpha$. The subcomplex of $\ERcompl$ consisting of all vertices and edges is the Erd\H{os}--R\'enyi graph $G(n,\alpha)$ where each edge is included independently with probability $\alpha$. Letting the scale $\alpha$ range from $0$ to $1$  yields the \define{Erd\H{o}s--R\'enyi filtration model $\ERfilt$}.  

Notice that the probability that two edges have the same length is zero. It is therefore of no harm to assume that each edge in $\ERfilt$ has different length. 
Since the Erd\H{o}s--R\'enyi filtration is a clique filtration, the order on edges and triangles induced by their entrance times yields a good order, even if arbitrary tie-breaks are performed. 
We will only consider $1$-dimensional boundary matrices since, to the best of our knowledge, no probabilistic results for higher dimensions are available. 

\paragraph{Probabilistic results.} As in the preceding two sections, a bound on $\beta_1(\ERcompl)$ is necessary. 
\begin{lemma} \label{lemma_er_beta}
There are constants $\kappa > 0$ and $c > 0$ such that if $\alpha > c \cdot \sqrt{\frac{\log\ver}{\ver}}$
\begin{equation*}
    \Prob{\betti_1(\ERcompl) > 0} < \kappa \cdot \ver^{- 4}\, .
\end{equation*}
\end{lemma}
This result is almost given in ~\cite[Theorem 1.2]{demarco2013triangle}~: they show that for $c=\frac{3}{2}$, $\betti_1(X)=0$ with high probability
(i.e., the probability goes to $1$ when $n$ goes to $\infty$). 
We will need the stated speed of convergence for our proof, and the proof of~\cite{demarco2013triangle}
yields this guarantee (with a constant larger than $3/2$). However, proving this requires us to go through a large part of the technical details
of that paper. 
We defer to~\cref{A_lemma_ER} for these details.

\cref{lemma_er_beta} is not quite sufficient to bound the probability of $\betti_1(\text{ER}(n,\rho_i)$ being non-zero. The reason is that $\text{ER}(n,\rho_i)$ is dependent on the entrance time of the $i$-th edge, not merely on the scope $\alpha$. However, we can derive a (crude) bound to relate the two concepts:
\begin{lemma}\label{lem:gnp_gnm_bound}
Let $\row=\binom{n}{2}$ and $\alpha=\frac{i}{\row}$. Then
\[
\ver^2\cdot\Prob{\betti_1(\ERcompl) > 0} \geq \Prob{\betti_1(\text{ER}(n,\rho_i) > 0} \, .
\]
\end{lemma}
\begin{proof}
We have that 
\begin{align*}
&\Prob{\betti_1(\ERcompl)> 0} \geq \Prob{(\betti_1(\ERcompl)> 0)\wedge \ERcompl\text{ has exactly i edges}}
\\
& = \Prob{(\betti_1(\ERcompl)> 0)| \ERcompl \text{ has exactly i edges}}\cdot\Prob{\ERcompl\text{ has exactly i edges}} \, .
\end{align*}
We claim that the first factor is equal to $\Prob{\betti_1(\text{ER}(n,\rho_i)> 0}$. 
Indeed, under the condition of having exactly $i$ edges, the underlying Erd\H{o}s--R\'enyi graph of $\ERcompl$ is just drawn uniformly at random among all $i$-edge graphs with $\ver$ vertices (because of the symmetry of the Erd\H{o}s--R\'enyi model), just as in the $i$-th complex of $\ERfilt$.

For the second factor, we observe that the number of edges is a binomial distribution whose expected value is equal to the integer $\row\cdot \alpha=i$.
It is known that, in this case, the probability is maximized at the expected value, see for instance~\cite{kaas1980mean}.
Hence, since there are $\row+1\leq \ver^2$ possible values for the distribution, we have that $\Prob{\ERcompl\text{ has exactly i edges}}\geq 1/\ver^2$.
\end{proof}
Combining these two statements with \cref{lem:main_lemma} yields our third main theorem:
\begin{mthm}\label{mainthm_er}
Let $\bdmat'$ be the matrix reduction of the $1$-dimensional boundary matrix of an Erd\H{o}s--R\'enyi filtration. Then
\[
\mathbb{E}[\fillin\bdmat']=\mathcal{O}(\ver^3\log \ver)
\]
and the cost of the matrix reduction is bounded by $\mathcal{O}(\ver^6\log \ver)$.
\end{mthm}
\begin{proof} Choose $\kappa$ and $c$ as in \cref{lemma_er_beta} and set $T\coloneqq c\cdot \row\sqrt{\frac{\log \ver}{\ver}}$. 
For every $i>T$, we have that $\alpha\coloneqq\frac{i}{\row}>c\sqrt{\frac{\log \ver}{\ver}}$.
Using \cref{lemma_er_beta} and \cref{lem:gnp_gnm_bound}, we can thus bound
\[
\Prob{\betti_1(\text{ER}(n,\rho_i) > 0}\leq \ver^2 \cdot \Prob{\betti_1(\ERcompl) > 0} \leq \frac{\kappa}{\ver^2}<\frac{\kappa}{\row}\, .
\]
Hence, the hypothesis of \cref{lem:main_lemma} is satisfied using the chosen $T$ and $A\coloneqq\kappa$.
It follows that
\[
\mathbb{E}[\fillin\bdmat']=\mathcal{O}((1+\kappa)\ver^2+\row^2\frac{\log \ver}{\ver})=\mathcal{O}(\ver^3\log \ver)
\]
proving the first part of the statement. 
The second part follows by \cref{lem:cost_bound} since the number of columns is $\mathcal{O}(\ver^3)$.
\end{proof}

\section{Comparison with experimental results}\label{S_experiments}

We ran experiments to compare the empirical outcome with our bounds
for fill-in and cost. For each filtration type, we generated $100$ random filtrations
for every considered value of $\ver$ and reduced their $1$-dimensional boundary matrices. 
We display the average fill-in and number of bitflips which corresponds to the cost of matrix reduction. 
We choose the bitflips as a proxy of complexity since, as discussed in \cref{subsec:Matrix_reduction}, the cost of the operation $M_i\leftarrow M_i+M_j$ can be realized using data structures for sparse matrices in $\mathcal{O}(\fillin M_j\log\fillin M_i)$ time, which matches bitflips up to a logarithmic factor.
Moreover, in all experiments, we use linear regression
on a log-log-scale to calculate the values $a$ and $b$ such
that the plot is best approximated by the curve $b n^a$.
Similar experiments have been performed in~\cite{SchreiberThesis}.

Figure~\ref{fig:cech_plot} shows the results for
\v{C}ech filtrations. For the fill-in (left figure), we observe
an empirical fill-in of $\Theta(n^{2.027})$
which is quite expected because of our upper bound of $\mathcal{O}(n^2\log^2)$ and a matching lower bound of $\mathcal{O}(n^2)$.
The cost (right figure) follows a curve of around $\Theta(n^{3.8})$ which is far from our upper bound
of $\mathcal{O}(n^5\log^2 n)$, suggesting that our bound on the cost is not tight.
This is perhaps not surprising because
our bound on the cost is based on the (pessimistic) assumption that the reduction of a column needs to add all previously reduced columns
of the matrix to it (see the proof of \cref{lem:cost_bound}). 
A tighter upper bound for the cost would have to improve on this part of the argument. 

The same reasoning holds true for the results for the Vietoris--Rips filtration (see Figure~\ref{fig:rips_plot}). 

\begin{figure}[H]
    \centering
    \includegraphics[width=7cm]{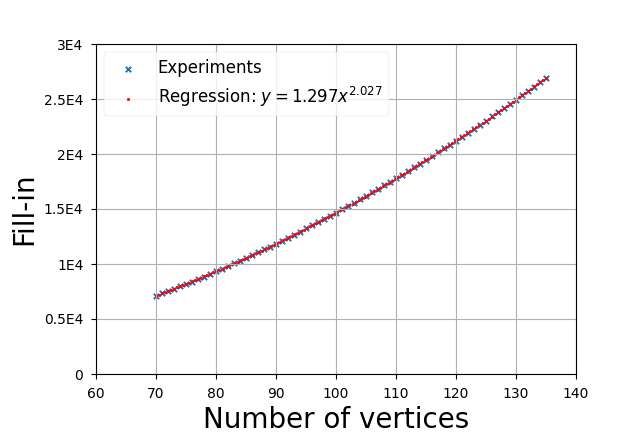}
    \includegraphics[width=7cm]{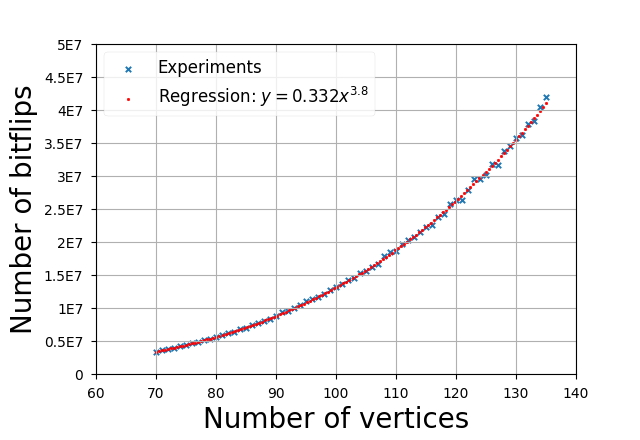}
    \caption{Average fill-in (left) and cost (right) for the reduction of the \v{C}ech filtration of a random point set sampled uniformly in $[-\frac{1}{2}, \frac{1}{2}]^3$. The regression coefficients are shown in the figure.}
    \label{fig:cech_plot}
\end{figure}

\begin{figure}[H]
    \centering
    \includegraphics[width=7cm]{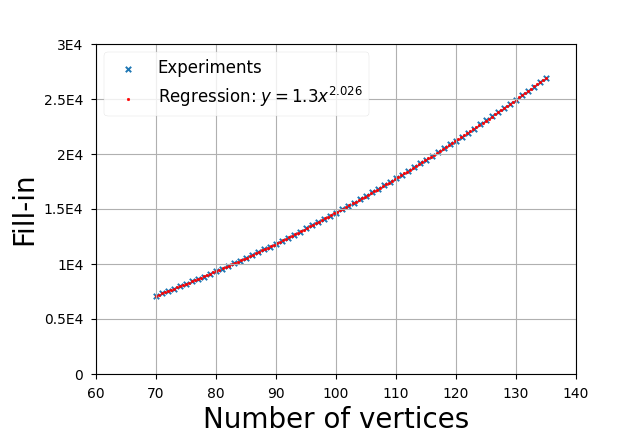}
    \includegraphics[width=7cm]{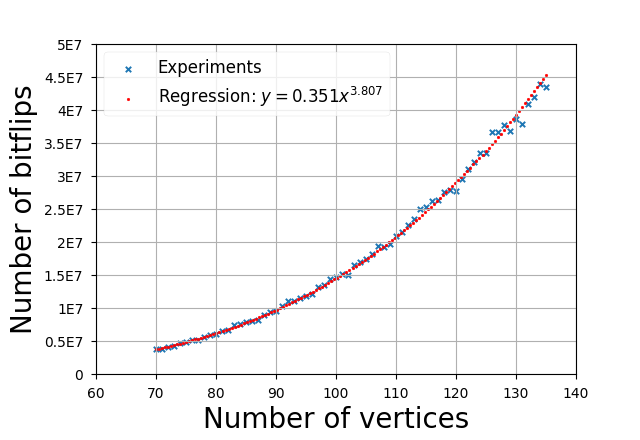}
    \caption{Average fill-in (left) and cost (right) for the reduction of the Vietoris--Rips filtration of a random point set sampled uniformly in $[-\frac{1}{2}, \frac{1}{2}]^3$. The regression coefficients are shown in the figure.}
    \label{fig:rips_plot}
\end{figure}

Figure~\ref{fig:er_plots} shows the results
for Erd\H{o}s--R\'enyi filtrations.
The regression yields an observed complexity of around $\mathcal{O}(n^{2.093})$ for the fill-in and $\mathcal{O}(n^{5.084})$ for the cost, which are quite far from
our upper bounds of $\mathcal{O}(n^3\log n)$ and $\mathcal{O}(n^6\log n)$, respectively.
Note that in the proof of Lemma~\ref{lem:main_lemma}, we assume that
all columns with a pivot smaller than the threshold $T$ are dense,
and we use the rather large value of $T=\Theta(\sqrt{n^3\log n})$ in the proof of Main Theorem~\ref{mainthm_er}. We speculate that a tighter bound has to
analyze the behavior in this ``subcritical regime'' more carefully
(a possible approach for that might be to use techniques from \cite{order_pivoting} to find a probabilistic bound on the density of the columns).
On the other hand, it is perhaps surprising that the empirical cost seems bigger than the empirical fill-in by a factor
very close to $n^3$. 
That suggests that, unlike in the Vietoris--Rips case, Lemma~\ref{lem:cost_bound} is not too pessimistic in bounding the cost of the reduction in this case.
\begin{figure}[H]
    \centering
    \includegraphics[width=7cm]{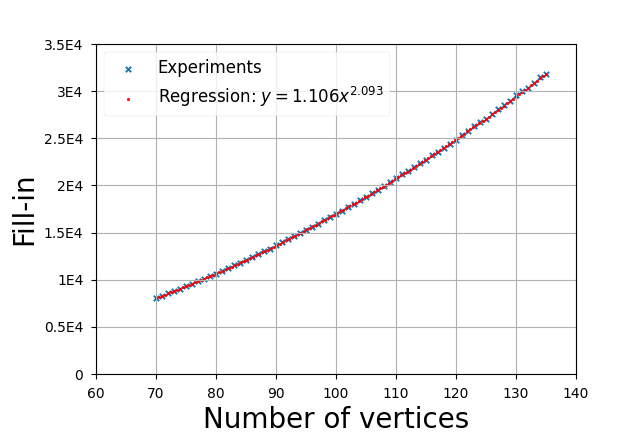}
    \includegraphics[width=7cm]{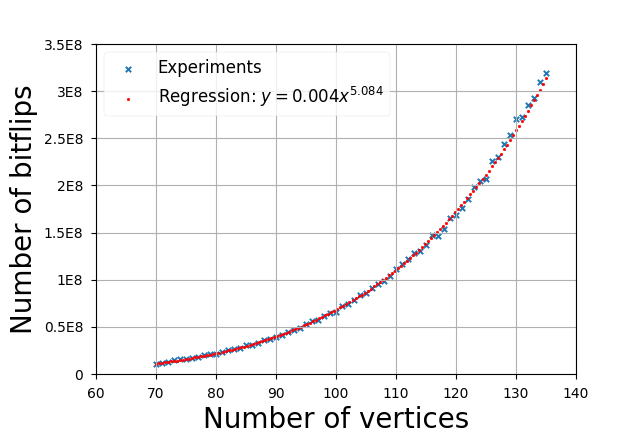}
    \caption{Average fill-in (left) and cost (right) for the reduction of the Erd\H{o}s--R\'enyi  filtration. The regression coefficients are shown in the figure.}
    \label{fig:er_plots}
\end{figure}

\section{Worst-case fill-in and complexity}\label{S_worst_case}
Our upper bounds on average fill-in and cost are smaller than the respective worst-case estimates.
However, these estimate are based on the assumption
that the reduction algorithm produces dense columns
(since the fill-in of a column with pivot $i$ is upper bounded with $i$).
Since the boundary matrix initially has only a constant number of non-zero entries per column, the question is whether such a bound is really achieved in an example, or whether the upper bound is not tight. 

Even for general boundary matrices of simplicial complexes, it requires some care to generate just one dense column in the reduced matrix. 
For the worst-case, however, one has to generate many such columns (to achieve the worst-case fill-in), and ensure that these columns get used in the reduction of subsequent columns (to achieve the worst-case cost). 
This has been done by Morozov~\cite{morozov2005persistence} for general simplicial complexes. 
However, restricting to clique complexes puts additional constraints and invalidates his example. 
In this section we show the following result:

\begin{theorem}
\label{thm:er_worst_case}
For every $n$, there is a clique filtration over $n$ vertices, for which the
left-to-right reduction of the $1$-boundary matrix has a fill-in of $\Theta(n^4)$ and a cost of $\Theta(n^7)$.
\end{theorem}

This result complements Main Theorem~\ref{mainthm_er} because the clique filtration we construct is a possible instance of the ER filtration model, hence the expected fill-in and cost for this model are indeed smaller than the worst-case by a factor of roughly $n$.

\paragraph{Idea of the construction.} 
Recall that a clique filtration is not completely fixed by the order of edges: 
many triangles can be created by the insertion of an edge, forming columns in the boundary matrix with the same pivot, and the order of these columns influences the resulting matrix (even if their order was irrelevant for the expected bounds). Our construction for Theorem~\ref{thm:er_worst_case} carefully chooses an edge order and an order of the columns with the same pivot. The details are rather technical, so we start with a more high-level description of the major gadgets of our construction.

The main idea is to define two groups of $\Theta(n^2)$ edges, or equivalently, rows of the boundary matrix that we call group II rows and group III rows, with group II rows having smaller index than group III rows (the notation is chosen to fit the technical description that follows). We first make sure to produce $\Theta(n^2)$ columns in the reduced matrix such that each column has exactly one non-zero group III element that is its pivot, and $\Theta(n^2)$ non-zero group II elements. We call such columns \define{fat} for now. Achieving this already yields a fill-in of $\Theta(n^4)$. To get the cost bound, we make sure to produce $\Theta(n^3)$ further columns (i.e., on the right of the fat columns) which we call \define{costly} columns. They have the property that during the reduction, they reach an intermediate state where they have gathered $\Theta(n^2)$ non-zero elements in group III and their pivot is in this group as well. To complete the reduction, the algorithm is than required to add $\Theta(n^2)$ fat columns to the current column. That means that the cost of reducing a costly column is $\Theta(n^4)$, and since there are $\Theta(n^3)$ costly columns, the bound on the cost follows. 

\begin{figure}
    \centering
\scalebox{.8}{ 
$
\arraycolsep=.5pt\def\arraystretch{.4}
\arrayrulecolor{black!20}
\begin{array}{|ccccc|cccccccc|ccccc|cccccc|ccccc}
\multicolumn{5}{c}{\text{\footnotesize{Cascade 1}}} & & & & & & & & &\multicolumn{5}{c}{\text{\footnotesize{Cascade 2}}}\\
& & & & & & & & & & & & & & & & & & & & & & & \\
\cline{1-24} 
\x & & & & & & & & & & & & & & & & & & & & & & \x & &\\
& & & & \x & & & & & & & & & & & & & & & & \x & & &  \textcolor{white}{a} & \textcolor{white}{a} & \textcolor{white}{a} & &\multirow{5}{*}{\large{II}}\\
& \x & & & & & & & & & & & & & & & & & & & & & &  \\
& & & \x & & & & & & & & & & & & & & & & \x & & & &	\\
& & \x & & & & & & & & & & & & & & & & & & & & & \x	\\
\cline{1-24} 
& & & & & & & & & \x & & & & \x & & & & & & & & & & \\
& & & & & & & \x & & & & & & & & & \x & & & \x & & & &  \textcolor{white}{a} & \textcolor{white}{a} & \textcolor{white}{a} & &\multirow{5}{*}{\large{III}}\\
& & & & & & \x & & & & & & & & \x & & & & & & & & & \x \\
& & & & & & & & & & & \x & & & & & & \x & & & & & \x & \\
& & & & & & & & & & & & \x & & & \x & & & & & \x & & & \\
\cline{1-24} 
\x & \x & & & & & & & & & & & \x & & & & & & & & & & & \\
& \x & \x & & & & & & & & & \x & & & & & & & & & & & &  \textcolor{white}{a} & \textcolor{white}{a} & \textcolor{white}{a} & &\multirow{5}{*}{\large{IV}}\\
& & \x & \x & & & & & & \x & & & & & & & & & & & & & & \\
& & & \x & \x & & & \x & & & & & & & & & & & & & & & & \\
& & & & \x & & \x & & & & & & & \x & & & & & & & & & & \\
\cline{1-24} 
& & & & & \x & \x & \x & & & & & & & & & & & & & & & &  \textcolor{white}{a} & \textcolor{white}{a} & \textcolor{white}{a} & &\multirow{3}{*}{\large{V}}\\
& & & & & & & & \x & \x & & & & & & & & & & & & & & \\
& & & & & & & & & & \x & \x & \x & & & & & & & & & & & \\
\cline{1-24} 
& & & & & & & & & & & & & \x & \x & & & & & & & & & \\
& & & & & & & & & & & & & & \x & \x & & & & & & & &  \textcolor{white}{a} & \textcolor{white}{a} & \textcolor{white}{a} & &\multirow{5}{*}{\large{VI}}\\
& & & & & & & & & & & & & & & \x & \x & & & & & & & \\
& & & & & & & & & & & & & & & & \x & \x & & & & \x & & \\
& & & & & & & & & & & & & & & & & \x & \x & & & & & \\
\cline{1-24} 
& & & & & & & & & & & & & & & & & & \x & \x & \x & & &  \textcolor{white}{a} & \textcolor{white}{a} & \textcolor{white}{a} & &\multirow{2}{*}{\large{VII}}\\
& & & & & & & & & & & & & & & & & & & & & \x & \x & \x \\
\cline{1-24}
\end{array}
$
}
\scalebox{.8}{ 
$
\arrayrulecolor{black!20}
\arraycolsep=.5pt\def\arraystretch{.4}
\begin{array}{|ccccc|cccccccc|ccccc|cccccc|}
 & & & & & \multicolumn{8}{c}{\text{\footnotesize{Fat col.}}} & & & & & & \multicolumn{6}{c}{\text{\footnotesize{Costly col.}}}\\
& & & & & & & & & & & & & & & & & & & & & & & \\
\hline 
\x & & & & & & \x & \x & & \x & & \x & \x & & & & & & & & & & \x & \\
& & & & \x & & \x & \x & & \x & & \x & \x & & & & & & & & & & &	\\
& \x & & & & & \x & \x & & \x & & \x & \x & & & & & & & & & & &	\\
& & & \x & & & \x & \x & & \x & & \x & \x & & & & & & & \x & & & & \\
& & \x & & & & \x & \x & & \x & & \x & \x & & & & & & & & & & & \x \\
\hline 
& & & & & & & & & \x & & & & \x & & & & & & \x & \x & & \x & \x \\
& & & & & & & \x & & & & & & & & & \x & & & \x & \x & & \x & \x \\
& & & & & & \x & & & & & & & & \x & & & & & \x & \x & & \x & \x \\
& & & & & & & & & & & \x & & & & & & \x & & \x & \x & & \x & \x \\
& & & & & & & & & & & & \x & & & \x & & & & \x & \x & & \x & \x \\
\hline 
\x & \x & & & & & & & & & & & \x & & & & & & & & & & & \\
& \x & \x & & & & & & & & & \x & & & & & & & & & & & & \\
& & \x & \x & & & & & & \x & & & & & & & & & & & & & & \\
& & & \x & \x & & & \x & & & & & & & & & & & & & & & & \\
& & & & \x & & \x & & & & & & & \x & & & & & & & & & & \\
\hline 
& & & & & \x & \x & \x & & & & & & & & & & & & & & & & \\
& & & & & & & & \x & \x & & & & & & & & & & & & & & \\
& & & & & & & & & & \x & \x & \x & & & & & & & & & & & \\
\hline 
& & & & & & & & & & & & & \x & \x & & & & & & & & & \\
& & & & & & & & & & & & & & \x & \x & & & & & & & & \\
& & & & & & & & & & & & & & & \x & \x & & & & & & & \\
& & & & & & & & & & & & & & & & \x & \x & & & & \x & & \\
& & & & & & & & & & & & & & & & & \x & \x & & & & & \\
\hline 
& & & & & & & & & & & & & & & & & & \x & \x & \x & & & \\
& & & & & & & & & & & & & & & & & & & & & \x & \x & \x \\
\hline
\end{array}
$
}
\caption{Example of a boundary matrix realising the worst-cases, unreduced (left) and halfway in the reduction (right), i.e. after running through both cascades.}
    \label{fig_highlight}
\end{figure}

The main question is: how do we produce fat and costly columns?
Let us start with fat columns.
The key notion is the one of the \define{cascade}; we refer to Figure~\ref{fig_highlight}
for an illustration of the following description.
We introduce another set of $\Theta(n^2)$ edges that define group IV (which come after group III in the edge order).
Let $i$ denote the row index of some group IV row.
Our construction ensures that there is a column with pivot $i$ that has as further entries $i-1$ and some entry in group II. We select this column as step column for $i$, so it does not change in the reduction. The set of these step columns forms the cascade.
Moreover, we ensure that all entries in group II over all cascade columns are at pairwise distinct indices to avoid unwanted cancellation in later steps. 

After construction the cascade, we include $\Theta(n)$ edges of group V. This generates $\Theta(n^2)$ columns that acquire a group IV pivot during the reduction. 
Moreover, we ensure that the (partially reduced) column has exactly one non-zero element of group III, and that all these group III indices are pairwise distinct for all columns
in this group.
To reduce this column further, we have to add the cascade columns, until the non-zero group III element becomes the pivot. While iterating through the cascade, the reduced column accumulates more and more non-zero elements in group II, resulting in a size of $\Theta(n^2)$. This creates the fat columns. Note that no two fat columns are added to each other because we ensure
that they have pairwise distinct pivots from group III~-- again, this avoids unwanted cancellation.

For generating costly columns, the idea is the same: we construct another cascade (using rows of group VI) and then a group VII to ensure that the cascade will fill up columns in the row indices of group III. Afterwards,
the reduction has to continue and adds columns with pivots at group III, which are precisely
the fat columns from the previous step.

\begin{figure}[H]
    \centering
\begin{tikzpicture}
\node (GI)[label={\textcolor{Violet}{{\bf I}}}] at (4,7.5) {};
\node (GII)[label={\textcolor{CornflowerBlue}{{\bf II}}}] at (6.7,3.85) {};
\node (GIII)[label={\textcolor{orange}{III}}] at (7.5,5.35) {};
\node (GIV)[label={\textcolor{OliveGreen}{IV}}] at (.5,3.1) {};
\node (GV)[label={\textcolor{red}{V}}] at (3.5,5.5) {};
    \tikzstyle{point}=[circle,draw=black,fill=black,inner sep=0pt,minimum width=4pt,minimum height=4pt]
    \node (roof)[point, label=above:{roof}] at (6,8.5) {};
\begin{scope}[xshift=50pt]
    \tikzstyle{point}=[circle,draw=black,fill=black,inner sep=0pt,minimum width=4pt,minimum height=4pt]
    \node (a0)[point, label=left:{$a_0$}] at (-2.1,5) {};
    \node (a1)[point, label=right:{$a_1$}] at  (-1.5,5) {};
    \node (adots)[label={$\cdots$}] at (-0.5,4.6) {};
    \node (ap)[point, label=right:{$a_{p-1}$}] at (0.1,5) {};
\end{scope}
\begin{scope}[xshift=50pt]
    \tikzstyle{point}=[circle,draw=black,fill=black,inner sep=0pt,minimum width=4pt,minimum height=4pt]
    \node (b0)[point, label=left:{$b_0$}] at (3,6.5) {};
    \node (b1)[point, label=right:{$b_1$}] at  (3.7,6.5) {};
    \node (bdots)[label={$\cdots$}] at (4.6,6.1) {};
    \node (bp)[point, label=right:{$b_{p-1}$}] at (5.3,6.5) {};
\end{scope}
\begin{scope}[xshift=50pt]
    \tikzstyle{point}=[circle,draw=black,fill=black,inner sep=0pt,minimum width=4pt,minimum height=4pt]
    \node (c0)[point, label=left:{$c_0$}] at (3,5) {};
    \node (c1)[point, label=right:{$c_1$}] at  (3.7,5) {};
    \node (cdots1)[label={$\cdots$}] at (4.6,4.6) {};
    \node (cp)[point, label=right:{$c_{p-1}$}] at (5.3,5) {};
\end{scope}
\begin{scope}[xshift=50pt]
    \tikzstyle{point}=[circle,draw=black,fill=black,inner sep=0pt,minimum width=4pt,minimum height=4pt]
    \node (d0)[point, label=below:{$d_0$}] at (3,3.5) {};
    \node (d1)[point, label=below:{$d_1$}] at  (3.7,3.5) {};
    \node (cdots2)[label={$\cdots$}] at (4.6,3.1) {};
    \node (dp)[point, label=below:{$d_{p-1}$}] at (5.3,3.5) {};
\end{scope}
\begin{scope}[xshift=50pt]
    \tikzstyle{point}=[circle,draw=black,fill=black,inner sep=0pt,minimum width=4pt,minimum height=4pt]
    \node (e0)[point, label=left:{$e_0$}] at (8,5) {};
    \node (e1)[point, label=right:{$e_1$}] at  (8.7,5) {};
    \node (edots)[label={$\cdots$}] at (9.6,4.6) {};
    \node (ep)[point, label=right:{$e_{p-1}$}] at (10.3,5) {};
\end{scope}
\draw[Violet, line width=0.05cm]  (roof) -- (a0);
\draw[Violet, line width=0.05cm]  (roof) -- (a1);
\draw[Violet, line width=0.05cm]  (roof) -- (ap);
\draw[Violet, line width=0.05cm]  (roof) -- (b0);
\draw[Violet, line width=0.05cm]  (roof) -- (b1);
\draw[Violet, line width=0.05cm]  (roof) -- (bp);
\draw[Violet, line width=0.05cm]  (a0) -- (a1);
\draw[Violet, line width=0.05cm]  (a0) to [out=-20,in=225] (c0);
\draw[Violet, line width=0.05cm]  (roof) -- (e0);
\draw[Violet, line width=0.05cm]  (roof) -- (e1);
\draw[Violet, line width=0.05cm]  (roof) -- (ep);
\draw[Violet, line width=0.05cm]  (e0) -- (e1);
\draw[Violet, line width=0.05cm]  (e0) to (bp);
\draw[Violet, line width=0.05cm]  (ep) -- (d0);
\draw[Violet, line width=0.05cm]  (ep) -- (d1);
\draw[Violet, line width=0.05cm]  (ep) -- (dp);
\draw[CornflowerBlue, line width=0.05cm]  (c0) -- (d1);
\draw[CornflowerBlue, line width=0.05cm]  (c0) -- (dp);
\draw[CornflowerBlue, line width=0.05cm]  (c1) -- (d0);
\draw[CornflowerBlue, line width=0.05cm]  (c1) -- (dp);
\draw[CornflowerBlue, line width=0.05cm]  (cp) -- (d0);
\draw[CornflowerBlue, line width=0.05cm]  (cp) -- (d1);
\draw[orange, thick=0.25]  (b0) -- (c0);
\draw[orange, thick=0.25]  (b0) -- (c1);
\draw[orange, thick=0.25]  (b0) -- (cp);
\draw[orange, thick=0.25]  (b1) -- (c0);
\draw[orange, thick=0.25]  (b1) -- (c1);
\draw[orange, thick=0.25]  (b1) -- (cp);
\draw[orange, thick=0.25]  (bp) -- (c0);
\draw[orange, thick=0.25]  (bp) -- (c1);
\draw[orange, thick=0.25]  (bp) -- (cp);
\draw[OliveGreen, thick=0.25]  (a0) to [out=-52,in=190] (d1);
\draw[OliveGreen, thick=0.25]  (a0) to [out=-48,in=210] (cp);
\draw[OliveGreen, thick=0.25]  (a1) to [out=-46,in=205] (cp);
\draw[OliveGreen, thick=0.25]  (a1) to [out=-50,in=190] (d0);
\draw[OliveGreen, thick=0.25]  (ap) to [out=-50,in=210] (c0);
\draw[red, densely dashdotted, thick=0.25]  (ap) to [out=57,in=140] (b0);
\draw[red, densely dashdotted, thick=0.25]  (ap) to [out=66,in=130] (b1);
\draw[red, densely dashdotted, thick=0.25]  (ap) to [out=70,in=130] (bp);
\end{tikzpicture}
\caption{Points of the five groups, where the point $x_i$ belong to the group $X$, and edges of the first five groups. The edges of group I and II are thicker, the ones of group V are dashed.}
\label{F:edge_construction}
\end{figure}

\begin{point}\label{P:groups_edges}
\emph{Groups of edges.}
Group I is given by the $3p$ edges between $roof$ and $A$, $roof$ and $B$, and $roof$ and $E$, by the $2(p-1)$ edges that form the path $(a_0,\dots,a_{p-1})$ in $A$ and the path $(e_0,\dots,e_{p-1})$ in $E$, by the edges $a_0c_0$ and $b_{p-1}e_0$, and finally by the $\frac{1}{2}(p^2+p)$ edges between the last $\frac{p+1}{2}$ vertices in $E$ and all the vertices in $D$.
Group II has $p(p-1)$ edges, given by all the edges of the complete bipartite graph between $C$ and $D$ but for the edges $c_id_{i}$, for $i=0,\dots,p-1$. 
Group III is given by all the $p^2$ edges that form a complete bipartite graph between $B$ and $C$. 
The order of the edges inside each of these groups is irrelevant and chosen randomly.
The groups IV and VI are given by a subset of cardinality $p^2-1$ of the $2p^2$ edges between $A$ and $C\cup D$, and between $E$ and $B\cup C$, respectively. 
These edges and their order have to be chosen carefully and we describe them in Paragraphs~\ref{P:order_edges_IV} and \ref{P:order_edges_VI}, respectively. 
Group V is given by all the $p$ edges from $a_{p-1}$ to $B$. 
Group VII is made of the $\frac{(p+1)p}{2}$ edges between $\{b_{\frac{p-1}{2}},\dots,b_{p-1}\}$ and $D$, and are ordered firstly in decreasing order on the indices in $B$ and then in decreasing order on the indices in $D$. 
Finally, the eighth group is given by all the remaining edges, whose order is irrelevant as long as they enter in the filtration after all the previous groups. 
We do not consider them further.
The first five groups are depicted in Figure~\ref{F:edge_construction} and the sixth and seventh groups in the zoom-in of Figure~\ref{F:zoom-in_edge_construction}.
\end{point}

\vspace{-1cm}
\begin{figure}[H]
    \centering
\scalebox{0.8}{
\begin{tikzpicture}
\node (GVI)[label={\textcolor{OliveGreen}{VI}}] at (4.3,2.4) {};
\node (GVII)[label={\textcolor{red}{\textbf{VII}}}] at (5.6,-1) {};
\begin{scope}[xshift=50pt]
    \tikzstyle{point}=[circle,draw=black,fill=black,inner sep=0pt,minimum width=4pt,minimum height=4pt]
    \node (b0)[point, label={$b_0$}] at (0,1.5) {};
    \node (bdots1)[label={$\cdots$}] at (0.8,1.1) {};
    \node (b1)[point, label={$b_{j}$}] at (1.6,1.5) {};
    \node (bdots2)[label={$\cdots$}] at (2.4,1.1) {};
    \node (bp)[point, label={$b_{p-1}$}] at (3.2,1.5) {};
\end{scope}
\begin{scope}[xshift=50pt]
    \tikzstyle{point}=[circle,draw=black,fill=black,inner sep=0pt,minimum width=4pt,minimum height=4pt]
    \node (c0)[point, label=left:{$c_0$}] at (0,0) {};
    \node (c1)[point, label=right:{$c_j$}] at  (1.6,0) {};
    \node (cp)[point, label=right:{$c_{p-1}$}] at (3.2,0) {};
\end{scope}
\begin{scope}[xshift=50pt]
    \tikzstyle{point}=[circle,draw=black,fill=black,inner sep=0pt,minimum width=4pt,minimum height=4pt]
    \node (d0)[point, label=below:{$d_0$}] at (0,-1.5) {};
    \node (ddots1)[label={$\cdots$}] at (0.8,-1.9) {};
    \node (d1)[point, label=below:{$d_j$}] at  (1.6,-1.5) {};
    \node (ddots2)[label={$\cdots$}] at (2.4,-1.9) {};
    \node (dp)[point, label=below:{$d_{p-1}$}] at (3.2,-1.5) {};
\end{scope}
\begin{scope}[xshift=50pt]
    \tikzstyle{point}=[circle,draw=black,fill=black,inner sep=0pt,minimum width=4pt,minimum height=4pt]
    \node (e0)[point, label=right:{$e_0$}] at (5,1) {};
    \node (edots1)[label={$\cdots$}] at (5.9,.6) {};
    \node (e1)[point, label=right:{$e_j$}] at (6.6,1) {};
    \node (edots2)[label={$\cdots$}] at (7.5,.6) {};
    \node (ep)[point, label=right:{$e_{p-1}$}] at (8.2,1) {};
\end{scope}
\draw[OliveGreen, thick=0.25]  (c1) to [out=40,in=170] (e0);
\draw[OliveGreen, thick=0.25]  (b0) to [out=40,in=135] (e0);
\draw[OliveGreen, thick=0.25]  (ep) to [out=145,in=32] ++ (-8.1,2.3) to [out=215,in=145] (c0);
\draw[OliveGreen, thick=0.25]  (b1) to [out=40,in=145] (ep);
\draw[OliveGreen, thick=0.25]  (bp) to [out=20,in=150] (ep);
\draw[red, line width=0.05cm]  (bp) to [out=-80,in=80] (dp);
\draw[red, line width=0.05cm]  (bp) to [out=265,in=30] (d1);
\draw[red, line width=0.05cm]  (bp) to [out=260,in=20] (d0);
\draw[red, line width=0.05cm]  (b1) to [out=-30,in=105] (dp);
\draw[red, line width=0.05cm]  (b1) to [out=220,in=80] (d0);
\draw[red, line width=0.05cm]  (b1) to [out=250,in=110] (d1);
\end{tikzpicture}
}
\caption{Zoom-in to the edges of groups VI and VII. For readability, $j= \frac{p-1}{2}$. The edges of group VII are thicker.}
\label{F:zoom-in_edge_construction}
\end{figure}
\begin{point}\label{P:order_edges_IV}
\emph{Group IV.}
Since each vertex in the graph $G=(C\cup D, \text{group II})$ has even degree, there exists an Eulerian path on $G$, starting in $c_0$. 
The edges in group IV are given as follows: for $j=1,\dots, p-1$, starting from the $(jp-j)$-th vertex in the path, we connect $p-1$ consecutive vertices of the Eulerian path to the vertex $a_{j-1}$. 
The edges are ordered first by increasing $j$ and then by the order of the Eulerian path.
Note that, by construction, none of the columns with pivots in row group IV have a non-zero element in row group III.

We choose as step columns the triangles with two edges in group IV, i.e., the elements in the cascade.
Of these triangles, all but $p-1$, i.e., the triangles constructed using the edges between elements in $A$, have as third edge an element of the Eulerian path. 
In particular, we have $(p-1)^2$ step columns that have all different elements in the row group II. 
These columns form the first cascade.
The order of the non-step triangles is irrelevant.
\end{point}

\begin{figure}[H]
    \centering
\begin{subfigure}{.45\textwidth}
\centering
\begin{tikzpicture}
\begin{scope}[xshift=50pt]
    \tikzstyle{point}=[circle,draw=black,fill=black,inner sep=0pt,minimum width=4pt,minimum height=4pt]
    \node (a0)[point, label=above:{$a_0$}] at (-2,5) {};
    \node (a1)[point, label=above:{$a_1$}] at  (-1,5) {};
    \node (ap)[point, label=above:{$a_{2}$}] at (0,5) {};
\end{scope}
\begin{scope}[xshift=50pt]
    \tikzstyle{point}=[circle,draw=black,fill=black,inner sep=0pt,minimum width=4pt,minimum height=4pt]
    \node (c0)[point, label=above:{$c_0$}] at (3,5) {};
    \node (c1)[point, label=above:{$c_1$}] at  (4,5) {};
    \node (cp)[point, label=above:{$c_{2}$}] at (5,5) {};
\end{scope}
\begin{scope}[xshift=50pt]
    \tikzstyle{point}=[circle,draw=black,fill=black,inner sep=0pt,minimum width=4pt,minimum height=4pt]
    \node (d0)[point, label=below:{$d_0$}] at (3,3.5) {};
    \node (d1)[point, label=below:{$d_1$}] at  (4,3.5) {};
    \node (dp)[point, label=below:{$d_{2}$}] at (5,3.5) {};
\end{scope}
\draw[Violet, line width=0.05cm]  (a0) -- (a1);
\draw[Violet, line width=0.05cm]  (a1) -- (ap);
\draw[Violet, line width=0.05cm]  (a0) to [out=-20,in=210] (c0);
\draw[CornflowerBlue, line width=0.05cm]  (c0) -- (d1);
\draw[CornflowerBlue, line width=0.05cm]  (c0) -- (dp);
\draw[CornflowerBlue, line width=0.05cm]  (c1) -- (d0);
\draw[CornflowerBlue, line width=0.05cm]  (c1) -- (dp);
\draw[CornflowerBlue, line width=0.05cm]  (cp) -- (d0);
\draw[CornflowerBlue, line width=0.05cm]  (cp) -- (d1);
\draw[OliveGreen, thick=0.25]  (a0) to [out=-50,in=190] (d1);
\draw[OliveGreen, thick=0.25]  (a0) to [out=50,in=140] (cp);
\draw[OliveGreen, thick=0.25]  (a1) to [out=52,in=150] (cp);
\draw[OliveGreen, thick=0.25]  (a1) to [out=48,in=150] (c1);
\draw[OliveGreen, thick=0.25]  (a1) to [out=-50,in=190] (d0);
\draw[OliveGreen, thick=0.25]  (ap) to (c0);
\draw[OliveGreen, thick=0.25]  (ap) to [out=-30,in=210] (c1);
\draw[OliveGreen, thick=0.25]  (ap) to [out=-40,in=170] (dp);
\end{tikzpicture}

\end{subfigure}
$ \qquad $
\begin{subfigure}{.4\textwidth}
\centering
\adjustbox{scale=0.8}{
\begin{blockarray}{ccccccccccc}
    \begin{block}{c[cccccccccc]}
    $a_0a_1$ &   &   & 1 &   &   &   &   &   &   \\
    $a_1a_2$ &   &   &   &   &   &   & 1 &   &   \\
    $a_0c_0$ & 1 &   &   &   &   &   &   &   &   \\
    \cline{1-11} 
    $c_0d_1$ & 1 &   &   &   &   &   &   &   &   \\
    $c_0d_2$ &   &   &   &   &   &   &   &   & 1 \\
    $c_1d_0$ &   &   &   &   &   & 1 &   &   &   \\
    $c_1d_2$ &   &   &   &   &   &   &   & 1 &   \\
    $c_2d_0$ &   &   &   & 1 & 1 &   &   &   &   \\
    $c_2d_1$ &   & 1 &   &   &   &   &   &   &   \\
    \cline{1-11} 
    $a_0d_1$ & 1 & 1 &   &   &   &   &   &   &   \\
    $a_0c_2$ &   & 1 & 1 &   & 1 &   &   &   &   \\
    $a_1c_2$ &   &   & 1 & 1 &   &   &   &   &   \\
    $a_1d_0$ &   &   &   & 1 & 1 & 1 &   &   &   \\
    $a_1c_1$ &   &   &   &   &   & 1 & 1 &   &   \\
    $a_2c_1$ &   &   &   &   &   &   & 1 & 1 &   \\
    $a_2d_2$ &   &   &   &   &   &   &   & 1 & 1 \\
    $a_2c_0$ &   &   &   &   &   &   &   &   & 1 \\
    \end{block}
    \end{blockarray}
    }
    \hspace{-.35cm}
    \adjustbox{scale=0.7}{
    \begin{turn}{270}
    \hspace{-5cm} \textcolor{Violet}{{\bf Group I}} \hspace{.5cm} \textcolor{CornflowerBlue}{{\bf Group II}} \hspace{1.5cm} \textcolor{OliveGreen}{Group IV} \hspace{1cm}
    \end{turn}
    }
\end{subfigure}
\caption{Example of the cascade construction for $p=3$ given by the edges (left) and relative (sub)matrix (right). For clarity's sake, we depicted only a subset of the edges, namely group II, IV and some of group I. The edges of group I and II are thicker.}
\label{F:example_groupIV}
\end{figure}

\begin{point}\label{P:order_triangle_V}
\emph{Step columns from group V.}
We fix as step columns all the triangles that have an edge in group V and two edges in group I. 
The order of all the other $p-1$ triangles created by that edge in V is irrelevant, and we add them as required by the filtration.
Note that none of the step columns has a non-zero element in row group III, but each of the non-step columns has one.
\end{point}

\begin{point}\label{P:order_edges_VI}
\emph{Edges of group VI and triangles from group VI and VII.}
Consider the subset $S$ of edges in the bipartite graph $B\cup C$ given by all the edges but $b_ic_i$, for $i=0,\dots,p-1$. 
Now every vertex in $B\cup C$ has even degree in $S$, and thus there exists an Eulerian path on $S$, starting in $b_{p-1}$.
We give the ordered edges of group VI analogously to how we gave the edges of group IV: for $j=1,\dots, p-1$, starting from the $(jp-j)$-th vertex in the path, we connect $p-1$ consecutive vertices of the Eulerian path to the vertex $e_{j-1}$. 
The edges are ordered first by increasing $j$ and then by the order of the Eulerian path.

For the order of the triangles closed by an edge in group VI, we choose as step columns the ones given by triangles with two edges in group VI and one in group III, but for the first step column which is given by the points $b_{p-1}c_{p-2}e_0$. 
Therefore the step columns form a cascade with all different elements in group III.

The step columns with pivots in group VII are given by the triangle with a vertex in $B$, one in $D$ and one in $E$. 
The order of the remaining triangles of both groups is irrelevant and chosen randomly.
\end{point}

\begin{proof}[Proof of Theorem~\ref{thm:er_worst_case}]
We begin by proving the fill-in, and then we use it to prove the complexity.
The edges of group I form some triangles whose reduction is not relevant to the worst-cases. 
The edges of group II and group III do not form any triangles. 
The edges of group IV close many triangles; the ones corresponding to the step columns form the cascade and are not reduced. 
The columns of the other triangles need to be reduced. 
At the end of their reduction, they will have a pivot in row group II which does not influence the rest of the construction. 
We now consider group V. 
By construction, in the reduction of a column $t$ with pivot in row group V, we add the previous step columns with pivot in row group V, moving the pivot of $t$ somewhere in the last $p-1$ rows of group IV. 
This triggers the cascade reduction, and we add all the $\Theta(p^2)$ columns of the cascade. 
All the cascade columns have different non-zero elements in row group II and none of them has a non-zero element in row group III.
Thus, $t$ accumulates $\Theta(p^2)$ non-zero entries in row group II before exiting the reduction with a pivot in row group III. 
This procedure has to be repeated for all the $\Theta(p^2)$ non-step columns that the edges of group V form, resulting in $\Theta(p^2)$ columns with $\Theta(p^2)$ elements, for a total fill-in of $\Theta(n^4)$.

We now discuss the complexity.  
We first note that, by construction, none of the columns with pivots in row group VI or VII has a non-zero element in group IV or V. 
Moreover, there are $\Theta(p^3)$ non-step columns with pivot in row group VII, given by one of the $\frac{p(p+1)}{2}$ edges in group VII and the $p-1$ points in $C$. 
Since the step columns with pivots in row group VII have each a non-zero element in the last half of row group VI, the reduction of those non-step columns pass through at least half of the cascade of group VI, thus accumulating $\Theta(p^2)$ elements in row group III. 
Now, from the fill-in discussion, the rows of group III are already pivots, specifically of the $\Theta(p^2)$-dense columns discussed above. 
Therefore, we have $\Theta(p^3)$ columns that accumulates $\Theta(p^2)$ elements, requiring thus an equal amount of operations, each of which flip $\Theta(p^2)$ elements, for a total of $\Theta(n^7)$ complexity.
\end{proof}
\section{Conclusions and future work}
\label{section_conclusion}
We established upper bounds for fill-in and cost of matrix reduction for three filtration types commonly studied in topological data analysis. 
In the \v{C}ech and Vietoris--Rips case, we managed to do this for arbitrary homological degree. 
Moreover, we showed that the fill-in bounds are tight for \v{C}ech and Vietoris--Rips filtrations in degree $k>1$.
In the Erd\H{o}s--R\'enyi case, the major obstacle to tackle dimensions $>1$ is the generalization of \cref{lemma_er_beta}, for which we are not aware of a proof (such a generalization is proved for homology
with rational coefficients~\cite{kahle-sharp}).

Complete filtrations were assumed in this work. 
It is common in practice to consider truncated filtrations. 
For instance, for Vietoris--Rips filtrations, one often removes all simplices with a diameter greater than a given threshold. In this case, when the number of columns is equal to $\column_0\leq\column$, the cost bound of Lemma~\ref{lem:cost_bound} reduces to $\column_0\fillin\bdmat'$. 

Our bounds on average fill-in and cost for matrix reduction are better than the currently best known worst-case bounds. 
We showed that these worst-cases can be realised in the Erd\H{o}s-R\'enyi model, but it is unclear if they can be realised by a \v{C}ech or Vietoris--Rips filtration.

In this work, points sampled uniformly i.i.d. in the unit cube were considered for the geometric filtrations. A natural direction is to investigate the case of points sampled close to a manifold embedded in Euclidean space. However, the gluing between the locally Euclidean patches may lead to the question of distributed computation of persistence. 

\paragraph{Declaration.} On behalf of all authors, the corresponding author states that there is no conflict of interest.

\paragraph*{Acknowledgments.} 
We thank Uli Bauer, Micka\"el Buchet, Matthew Kahle, and Jeff Kahn for useful discussions. We also thank the anonymous reviewer of a previous version of this article
whose insightful question led to an improved Main Theorem 1.
BG and MK were supported by the Austrian Science Fund (FWF) grant number P 29984-N35 and P 33765-N. MS was supported by the Austrian Science Fund (FWF) grant number W1230. 
GH was partially supported by {\'E}cole Polytechnique and thanks TU Graz for the hospitality during his Master Internship.

\printbibliography

\appendix

\section{Probabilistic part of Section~\ref{section_cech}}
\label{A_lemma_C}

\paragraph{Proof of Lemma \ref{lem:CentralLemmaB_CechCase}.} We repeat the proof of \cite{kahle2011random} Theorem $6.1$ for the convenience of the reader. Recall our adaptation of the result:  
\begin{lemma-non} 
Given positive integers $\ell$ and $k$, there exists a constant $c>0$ such that for  $\alpha \geq c {\left(\frac{\log n}{n}\right)}^{1/d}$:
\begin{equation*}
\mathbb{P} \left( \beta_k(\check{C}(\mathcal{X}_n,\alpha)) \neq 0 \right) \leq \frac{1}{n^\ell}
\end{equation*}
for sufficiently large $n$. 
\end{lemma-non}

\begin{proof}

The proof relies on the following Nerve Theorem (Theorem 10.7 in \cite{bjoerner}).
\begin{theorem}
If $X$ is a triangulable topological space and if $\mathcal{A}=(A_i)_{i\in I}$ is a finite cover of $X$ by closed sets such that every nonempty intersection of sets in $\mathcal{A}$ is contractible, then $X$ and the nerve $\mathcal{N}(\mathcal{A})$ are homotopy equivalent. 
\end{theorem}

Thus, whenever $\alpha$ is large enough such that the balls $\{B_{\alpha}(X_i) \}_{X_i\in \mathcal{X}_n}$ cover the unit cube $[-\frac{1}{2},\frac{1}{2}]^d$, the \v{C}ech complex $\cechmod$ is contractible and therefore the $k$-th Betti number is zero. Assume that there exists a radius $R\geq c (\log n/n)^{1/d}$ such that $\beta_k(\check{C}(\mathcal{X}_n,R)) \neq 0$. 
By contraposition, the balls $\{B_R(X_i)\}_{x_i\in \mathcal{X}_n}$ cannot cover the unit cube. 
We can see $R$ as a function depending on $n$. 
As such, $R$ is sufficiently small, i.e. $R\rightarrow 0$ as $n\rightarrow \infty$, and the probability of the balls not covering the cube is now bounded.  \\
Let $\lambda \mathbb{Z}^d$ denote the $d$-dimensional cubical lattice linearly scaled in every direction by the factor $\lambda=R/(2\sqrt{d})$. The Lebesgue measure is denoted by $\mu$. There are $N=\mu([-\frac{1}{2},\frac{1}{2}]^d)/\lambda^{d}+\mathcal{O}(1/\lambda^{d-1})$ boxes in $\lambda \mathbb{Z}^d$ intersecting $[-\frac{1}{2},\frac{1}{2}]^d$ and we write $S$ for the set of boxes completely contained in ${[-\frac{1}{2},\frac{1}{2}]}^d$. Notice that each box intersecting the boundary of ${[-\frac{1}{2},\frac{1}{2}]}^d$ is adjacent to at least one box in $S$. 
If every box in $S$ contains at least one point in $\mathcal{X}_n$, since each box has a diameter of $\lambda \sqrt{d} = R/2$, the unit cube is completely covered by balls of radius $R$ centered in points of $\mathcal{X}_n$.  \\
By the two paragraphs above, the considered event implies that there has to be at least one box in $S$ which contains no points of $\mathcal{X}_n$. 
With $p_0\coloneqq\mathbb{P}(B\cap \mathcal{X}_n=\emptyset)$ being the probability that box $B\in S$ is empty, we can write

\begin{equation}
\label{eq:CechProof_Estimate1}
\mathbb{P} \left( \exists \alpha \geq c {\left(\frac{\log n}{n}\right)}^{1/d} \ \mathrm{such} \ \mathrm{that} \ \beta_k(\check{C}(\mathcal{X}_n,r)) \neq 0 \right) \leq \mathbb{P}(\exists B \in S\colon B\cap \mathcal{X}_n= \emptyset) \leq N p_0 \, .
\end{equation}
For a box $B\in S$ the independence and identical uniform distribution of the points in $\mathcal{X}_n$ yields
\begin{equation}
\label{eq:CechProof_Estimate2}
p_0 = \left(1- \frac{\lambda^d}{\mu({[-\frac{1}{2},\frac{1}{2}]}^d)} \right)^n \leq \exp ( -\lambda^d n) = \exp(-CR^dn)
\end{equation}
where $C = \frac{1}{2^d d^{d/2}}$. With inequality $R\geq c {(\log n/n)}^{1/d}$ we obtain
\begin{equation*}
p_0 \leq \exp(- C c^d \log(n)) = n^{-C c^d}. 
\end{equation*}
Further, $N=1/\lambda^d + \mathcal{O}(1/\lambda^{d-1})=(1+o(1))/(CR^d)$ and thus with setting $c = {((1+\ell)/C )}^{1/d}$, 
\begin{equation*}
N p_0 \leq \frac{1 + o(1)}{C c^d \log n} n^{1-C c^d}\leq \frac{1}{n^\ell}
\end{equation*}
for sufficiently large $n$ which gives the desired result.
\end{proof}

\paragraph{Proof of Lemma \ref{lem:SecondLemmaB_CechCase}.} The result to be proven is again recalled: 

\begin{lemma-non}
Let $\mathcal{X}_n\coloneqq\left\{ p_1,\dots,p_n \right\}$ be a set of i.i.d. uniformly distributed random variables in ${[-\frac{1}{2},\frac{1}{2}]}^d$ and define $\alpha^*\coloneqq c{\left(\frac{\log n}{n}\right)}^{1/d}$. Let $N_k=N_k(\mathcal{X}_n)$ be the number of $k$-simplices in $\check{C}(\mathcal{X}_n,\alpha^*)$. 
Then for any $\ell\in\mathbb{N}$ there exists a constant $\lambda >0$ such that we have
\begin{equation}
\label{eq:conc_ineq}
\mathbb{P}\left(N_k\geq \lambda n \log^k n\right) \leq \frac{1}{n^\ell},
\end{equation}
for sufficiently large $n$. 
\end{lemma-non}

\begin{proof}
	We prove the claim by induction on $k$. 
	For $k=0$, since $N_0=n$ always, we can	set $\lambda\coloneqq 2$ to satisfy the inequality.
	For arbitrary $k\geq 1$, by induction, we can choose $\lambda_{k-1}$ such that
	\[\mathbb{P}(N_{k-1}\geq\lambda_{k-1} n \log^{k-1} n)\leq \frac{1}{n^{\ell+1}}.\]

        Informally, this means that having many $(k-1)$-simplices
        is an unlikely event. Now, to have
        many $k$-simplices, we either need many $(k-1)$-simplices already,
        or we have to generate many $k$-simplices out of not-so-many
        $(k-1)$-simplices. Formally:
	
	\begin{equation}\label{eq:lemma4_4eq_1}
		\begin{aligned}
			\mathbb{P}\left(N_k\geq \lambda_k n\log^k n\right) 
			&= \mathbb{P}\left(N_k\geq \lambda_k n \log^k n \wedge N_{k-1} \geq \lambda_{k-1} n \log^{k-1}n\right)  \\
			&\quad + \mathbb{P}\left(N_k\geq \lambda_k n \log(n)^k \wedge N_ {k-1} < \lambda_{k-1} n \log^{k-1}n\right) \\
			&\leq \mathbb{P}\left(N_{k-1} \geq \lambda_{k-1} n \log^{k-1}n\right) \\
			&\quad + \mathbb{P}\left(N_k\geq \lambda_k n \log^kn \wedge N_{k-1} < \lambda_{k-1} n \log^{k-1}n\right) \, .
		\end{aligned}
	\end{equation}
	
	The first summand in the last inequality of (\ref{eq:lemma4_4eq_1}) is bounded by induction, so we turn our attention to the second summand. 
	Let $\sigma$ be a $k$-subset of $[n]$, and $V_\sigma\coloneqq\{p_i\mid i\in\sigma\}$ the corresponding set of vertices, forming a
        $(k-1)$-simplex.
	Let $N_\sigma$ denote the random variable given by the number of $k$-simplices that contain $V_\sigma$ as a face. 
	It then follows that
	\[
	N_k=\frac{1}{k+1}\sum_{\sigma}N_\sigma \, .
	\]	
	A priori, there are $\binom{n}{k}$ choices for $\sigma$.
	However, $N_\sigma>0$ only if $V_\sigma$ is a $(k-1)$-simplex in the \v Cech complex. 
	Therefore, by the second condition in the event $(N_k\geq \lambda_k n \log(n)^k \wedge N_ {k-1} < \lambda_{k-1} n \log(n)^{k-1})$, the sum
	only has $\lambda_{k-1} n (\log n)^{k-1}$ non-zero summands. 
	By the first condition in the said event, one summand has to be ``large'' in the following sense:
	
	\begin{equation}\label{eq:lemma4_4eq_2}
		\begin{aligned}
			\mathbb{P}\left(N_k\geq \lambda_k n \log^kn \wedge N_{k-1} < \lambda_{k-1} n \log^{k-1}n\right)
			&\leq \mathbb{P}\left(\exists \sigma\colon N_\sigma\geq (k+1)\frac{\lambda_k n(\log n)^k}{\lambda_{k-1} n (\log n)^{k-1}}\right)
			\\
			&= \mathbb{P}\left(\exists \sigma\colon N_\sigma\geq (k+1)\frac{\lambda_k}{\lambda_{k-1}}\log n\right)
			\\
			&\leq \sum_{\sigma} \mathbb{P}\left(N_\sigma\geq (k+1)\frac{\lambda_k}{\lambda_{k-1}}\log n\right)
			\\
			&\leq n^k \mathbb{P}\left(N_{[k]}\geq (k+1)\frac{\lambda_k}{\lambda_{k-1}}\log n\right)
		\end{aligned}
	\end{equation}
	Here we have used the union bound and the fact that the probability
	of $N_\sigma$ is the same for every $\sigma$ because the points are drawn
	independently. 
	
	In order to bound the last term in (\ref{eq:lemma4_4eq_2}), we notice that, for $N_{[k]}$ to be large, we have to sample many points close to $p_1$:
	Every $k$-simplex that contains $p_1,\ldots,p_k$ contains one further
	point $p_m$ with $m>k$, and $p_m$ has a distance of at most $2\alpha$
	to $p_1$, as otherwise, there would not be an edge $p_1p_m$ in the \v Cech complex.
	Therefore, we have
	\begin{align*}
		\mathbb{P}(N_{[k]}\geq (k+1)\frac{\lambda_k}{\lambda_{k-1}}\log n)
		&\leq  \mathbb{P}(\text{at least }(k+1)\frac{\lambda_k}{\lambda_{k-1}}\log n \text{ points in }\{p_{k+1},\ldots,p_n\}\text{ lie in } B_{2\alpha}(p_1))
		\\
		&\leq \mathbb{P}(\text{at least }(k+1)\frac{\lambda_k}{\lambda_{k-1}}\log n \text{ points among }n-k\text{ i.i.d. points lie in a $2\alpha$-ball})  
	\end{align*}
	
	The last line corresponds to a sum of i.i.d. Bernoulli random variables. The success probability of each of them
	is the volume of the $2\alpha$-ball $B_{2\alpha}$ given by $2^dc^d\mu\frac{\log n}{n}$, where $\mu$
	denotes the volume of the $d$-dimensional unit ball. The expectation of the sum is then $C\log n$ for some constant $C$. This situation allows for the application of a Chernoff bound (Eq. $1.$ in Theorem $4.4$ in \cite{mitzenmacher2005probability}, applied with the inequality $\frac{2\delta}{2+\delta}\leq \log(1+\delta)$). 
	We get
	\[
	\mathbb{P}\left(\text{at least $(1+\delta)C\log n$ points are sampled in }B_{2\alpha}\right) 
	\leq e^{-\frac{\delta^2C\log n}{2+\delta}}
	\leq n^{-\delta^2C'}
	\leq n^{-(\ell+k+1)}
	\]
	for $\delta$ chosen appropriately and $C'$ is a constant.
	
	We now chose $\lambda_k$ such that $(k+1)\frac{\lambda_k}{\lambda_{k-1}}\geq (1+\delta)C$. 
	Thus, putting everything together:
	
	\begin{align*}
		\mathbb{P}\left(N_k\geq \lambda_k n \log^k n \wedge N_{k-1} < \lambda_{k-1} n \log^{k-1}n\right)&
		\leq n^k \mathbb{P}(N_{[k]}\geq (k+1)\frac{\lambda_k}{\lambda_{k-1}}\log n)\\
		&\leq n^k \mathbb{P}\left(\text{at least $(1+\delta)C\log n$ points are sampled in }B_{2\alpha}\right)\\
		&\leq n^k n^{-(\ell+k+1)}\\
		&= n^{-(\ell+1)}.
	\end{align*}
	This implies that
	\[
	\mathbb{P}(N_k\geq\lambda_k n\log^k n) \leq\frac{1}{n^{\ell+1}} + \mathbb{P}\left(N_k\geq \lambda_k n \log^k n \wedge N_{k-1} < \lambda_{k-1} n \log^{k-1}n\right)\leq \frac{1}{n^{\ell+1}}+\frac{1}{n^{\ell+1}}\leq\frac{1}{n^\ell} \qedhere
	\]
\end{proof}

\section{Proof of Lemma~\ref{lemma_rips_beta}}\label{A_lemma_VR}
In this section we will derive our adaptation of \cite{kahle2011random} Theorem $6.2$. 
Recall the statement:
\begin{lemma-non} 
Given positive integers $\ell$ and $k$, there exists a constant $c>0$ such that for  $\alpha \geq c {\left(\frac{\log(n)}{n}\right)}^{1/d}$:
\begin{equation}
\mathbb{P} \left( \beta_k(\text{VR}(\mathcal{X}_n,\alpha)) \neq 0 \right) \leq \frac{1}{n^\ell}
\end{equation}
for sufficiently large $n$. 
\end{lemma-non}

\begin{proof}

The proof relies on discrete Morse theory. The fundamental theorem of discrete Morse theory (see Theorem $2.5$ of \cite{forman2002user}) says that a simplicial complex with a discrete gradient vector field $V$ is homotopy equivalent to a CW complex with one cell of dimension $k$ for each pivotal $k$-dimensional simplex. 
The defintion of cellular homology gives that $\beta_k$ is smaller than the number of $k$-cells in the given CW complex. 
Thus, if for a constant $c>0$ and $\alpha>c{\left(\frac{\log n}{n}\right)}^{1/d}$ a given discrete gradient vector field on $\VRcompl$ has no pivotal simplices, then $\beta_k(\VRcompl)=0$. 
By contraposition, $\beta_k(\VRcompl)\neq 0$ implies that there exists a pivotal simplex in the Vietoris--Rips complex. 
The probability of this event will be bounded in the following. 

First a definition from discrete Morse theory: A \define{discrete vector field} $V$ of a simplicial complex $K$ is a collection of pairs of faces $\{ \sigma \subseteq \tau\}$ of $K$ such that each face is in at most one pair. A closed $V$-path in a discrete vector field is then a sequence of faces 
\begin{equation*}
\sigma_0\subseteq \tau_0 \supseteq \sigma_1 \subseteq \tau_1 \supseteq\cdots \subseteq \tau_n \supseteq \sigma_{n+1}
\end{equation*}
with $\sigma_{i+1}\neq \sigma_i$ such that $\{\sigma_i\subseteq \tau_i\}\in V$ for $i=0,\dots,n$ and $\sigma_{n+1}=\sigma_0$. $V$ is a \define{discrete gradient vector field} is there are no closed $V$-paths. 
Any simplex $\nu\in K$ not in any pair of $V$ is called \define{critical}. 

We start by indexing the points in $\mathcal{X}_n$ by distance to the origin, i.e. $\lVert X_1 \rVert < \lVert X_2 \rVert < \cdots < \lVert X_n \rVert$. Note that no two points have that same distance almost surely. Now define a discrete vector field $V$ on $\VRcompl$ by - whenever possible - pairing a face $S=\{X_{i_1},X_{i_2},\dots,X_{i_k}\}$ with a face $\{X_{i_0}\}\cup S$ with $i_0<i_1$ and $i_0$ as small as possible. $V$ is well defined as each face gets paired at most once: A face $S$ cannot be paired to two higher dimensional faces $\{X_a\}\cup S$ and $\{X_b\}\cup S$ as the pairing will be performed with the smaller index $\mathrm{min}(a,b)$. It is also not possible for $S$ to get paired with both a higher and a lower dimensional face: Suppose that $S$ gets paired with $\{X_a\}\cup S$. Then $\lVert X_a \rVert < \lVert X \rVert$ for every $X\in S$. In this case, no codimension $1$ face $F\subseteq S$ can be paired with $S$, since the pairing $\{X_a\}\cup F$ would be preferred. $V$ is even a discrete gradient vector field because the indices are decreasing along any $V$-path. 

Given a $k$-dimensional simplex in $\VRcompl$, $\sigma=\{X_{i_1},\dots,X_{i_{k+1}}\}$, this  simplex is critical, that is, unpaired in $V$ if
\begin{enumerate}[label=(\roman{enumi})]
\item there is no common neighbor $X_a$ with $a<i_1$ to the vertices of $\sigma$, or else the simplex would be paired up by adding such a point with smallest index, or
\item $\sigma$ would be paired up with $\{X_{i_2},\dots,X_{i_{k+1}}\}$ unless $\{X_{i_2},\dots,X_{i_{k+1}}\}$ has a common neighbor with index smaller than $i_2$. 
\end{enumerate}

Assume now that $X_{i_0}$ is a common neighbor of $\sigma$ as in in $(i)$ and let $\alpha \geq c {\left(\frac{\log(n)}{n}\right)}^{1/d}$ where $c>0$ is a constant defined later, but depending only on $k,\ell$ and $d$. Then $\lVert X_{i_1}\rVert \geq \frac{1}{2}\alpha$, as otherwise $\lVert X_{i_0}-X_{i_1}\rVert < \alpha$ or $\lVert X_{i_0} \rVert > \lVert X_{i_1} \rVert$ would contradict our assumptions. Further, $\lVert X_{i_0}-X_{i_1}\rVert > \alpha$ and $\lVert X_{i_j}-X_{i_t}\rVert \leq \alpha$ for $0\leq j < t \leq k+1$. Then, by the (technical) Lemma $5.3$ in \cite{kahle2011random}, the Lebesgue measure of the intersection
\begin{equation*}
I=\bigcap_{j=1}^{k+1}B_\alpha(X_{i_j})\cap B_{\lVert X_{i_1} \rVert}(0)
\end{equation*}
can be bounded: $\mu(I)\geq \epsilon_d \alpha^d$ with $\epsilon_d>0$ depending only on $d$, i.e. being a constant for us. 
If any vertices fall into the interseciton $I$, then this vertex would be a common neighbor of $\sigma$ with index smaller than $i_1$ and thus $\sigma$ would be paired. 
If $\sigma$ is critical, then $\sigma$ is unpaired and thus $(i)$ has to be satisfied, i.e. no vertices of $\VRcompl$ may lie in $I$. 
Since the points of $\mathcal{X}_n$ are i.i.d. uniformly sampled from the unit cube $[\frac{1}{2},\frac{1}{2}]^d$, the probability of this event is $(1-\mu(I))^{n-k-2}$. 
We have: 
\begin{equation*}
\begin{split}
&\mathbb{P}(\beta_k(\VRcompl)\neq 0) \leq \mathbb{P}(\text{there exists a critical k-simplex }\sigma\in \VRcompl)\\ &\leq \binom{n}{k+1}\left(1-\mu(I)\right)^{n-k-2} \leq \binom{n}{k+1}\left(1-\epsilon_d\alpha^d \right)^{n-k-2} \leq \binom{n}{k+1}\exp \left(-\epsilon_d \alpha^d (n-k-2) \right)\\&\leq Cn^{k+1}\exp \left(-\epsilon_d c^d \frac{\log(n)}{n}(n-k-2) \right) \leq \frac{1}{n^\ell}
\end{split}
\end{equation*}
for $c=(\frac{l+k+1}{\epsilon_d})^{1/d}$ and large enough $n$. 
\end{proof}

\section{Proof of Lemma~\ref{lemma_er_beta}}\label{A_lemma_ER}
In this section, we prove \cref{lemma_er_beta}, that we restate as \cref{lemma_er_beta_appendix}. 
This result is very close to the one stated in \cite[Theorem 1.2]{demarco2013triangle}.
The main difference is the precision of the bound on $\Prob{\beta_1(\ERcompl) > 0}$: in our case, we need a bound of $\mathcal{O}(\ver^{-4})$.
Nevertheless, the proof given in  \cite{demarco2013triangle} can be slightly adapted to give the desired bound, at the expense of increasing $c$. We first adapt our notation. Remember, that $\ERcompl$ is the clique complex induced by the Erd\H{o}s-R\'enyi graph $G(\ver,\alpha)$. We will write $X=X(G)=\ERcompl$ and $p:=\alpha$ for the rest of this section.
\begin{lemma} \label{lemma_er_beta_appendix}
There are constants $\kappa > 0$ and $c > 0$ such that if $p > c \cdot \sqrt{\frac{\log\ver}{\ver}}$, $G \sim G(\ver,p)$ and $X$ is the clique complex of $G$, then:
\begin{equation*}
    \Prob{\betti_1(X) > 0} < \kappa \cdot \ver^{- 4}\, .
\end{equation*}
\end{lemma}
\begin{proof}[Sketch of the proof]
Since 
\[
\Prob{\betti_1(X) > 0} = \Prob{\Cspace \neq \Tspace} = \Prob{\Cspace^\perp \neq \Tspace^\perp} \, 
\] 
where $\Cspace$ and $\Tspace$ denote the cycle and the triangle space of $G$, respectively, and $^\perp$ denote the orthogonal with respect to the usual inner product, we will prove the bound for the right-most probability.
To do this, we need two lemmas (originally stated in \cite[Section 3]{demarco2013triangle}): \cref{lemma_3_1} guarantees that, under some assumptions on $G$, it is highly unlikely that there is a small element in $\Tspace^\perp \setminus \Cspace^\perp$; \cref{lemma_3_3} gives us a property which is satisfied by every graph that is large enough. 

Finally, we show that if there is a large element $F$ in $\Tspace^\perp \setminus \Cspace^\perp$, then we can build a graph $B$ that is very likely to satisfy the hypotheses of \cref{lemma_3_3} but is very unlikely to satisfy its thesis.
Therefore, $F$ is unlikely to exist, and the claim follows.
\end{proof}

Throughout this section, we set $G \sim G(\ver,p)$, and we use the following notations:
\begin{itemize}
\item $V$ is the set of vertices of $G$ and $\fullG{\ver}$ is the clique whose vertices are the $\ver$ elements of $V$;
\item The number of edges in the graph $H$ is denoted by $|H|$;
\item For $x,y \in V$, $\neigh{H}{x}$ is the neighborhood of $x$ in $H$, $\neigh{H}{x,y}\coloneqq \neigh{H}{x} \cap \neigh{H}{y}$ is the set of common neighbors of $x$ and $y$, and we write $d_H(x) = |\neigh{H}{x}|$ and $d_H(x, y) = |\neigh{H}{x, y}|$;
\item If $S$ and $T$ are two sets of vertices, $\cut{S, T}$ is the set of edges joining $S$ and $T$ in $G$, i.e., the set of edges that have one end in $S$ and the other in $T$. We also write $\cut{S} = \cut{S, V - S}$ and $\cut{v} = \cut{H,\{ v \}}$ for any $v \in V$;
\item $\nbtri{H}$ is the number of triangles of $H$.
\end{itemize}
Moreover, if we do not specify the graph $H$ in the above list, we mean $H=G$.
\medskip

For ease of reference, we now display a series of results initially proved in \cite{demarco2013triangle}, which all follow from Hoeffding inequalities for binomial laws (see \cite[Theorem 2.1]{chernoff_ref}).
We adapt their notation to our goals, namely by setting their $o$ bounds as precise values for the inequalities, which are then true when $\ver$ is large enough. 
These results will be used both in the proof of \cref{lemma_3_1} and in the proof of \cref{lemma_er_beta_appendix}, where they will give us the necessary bounds provided that we choose fixed values for $\epsilon$ that are small enough.
\medskip

The first proposition claims, in particular, that the probability that an ER graph has ``many'' edges is low.
\begin{proposition}[Proposition 2.3 of \cite{demarco2013triangle}] \label{proposition_2_3}
For every $\epsilon > 0$,
\begin{equation*}
    \Prob{\abs{|G| - \frac{\ver^2p}{2}} > \epsilon \frac{\ver^2p}{2}} \leq \exp \left(- \frac{\epsilon^2}{8} \ver^2p \right),
\end{equation*}
\begin{equation*}
    \Prob{\exists v \in V \text{ such that } \abs{d(v) - \ver p} > \epsilon \ver p} \leq \ver \cdot \exp \left(- \frac{\epsilon^2}{4} \ver p \right)\, .
\end{equation*}
Moreover, for some $\gamma > 0$,
\begin{equation*}
    \Prob{\exists v, w \in V \text{ such that } d(v, w) \geq 4 \ver p^2} \leq \exp \left(- \gamma \ver p^2 \right)\, .
\end{equation*}
\end{proposition}

The following proposition shows two statements. First, that the event, that in an ER graph there are two big, disjoint sets of vertices with many more or many fewer edges between them than expected has low probability. 
Second, that such an estimate can be extended to a set of vertices in $G$ of arbitrary size as well as its complement. 
\begin{proposition}[Proposition 2.4 of \cite{demarco2013triangle}] \label{proposition_2_4}
For each $\delta > 0$, there is a $K > 0$ such that
\begin{align*}
    & \Prob{\exists S, T \subset V \text{ such that } \Big[ S \cap T = \emptyset \Big]  \land \Big[ |S|, |T| > Kp^{-1} \log\ver \Big] \land \Big[ \abs{|\cut{S, T}| - |S||T|p} > \delta |S||T|p \Big]} \\ 
    & \qquad \leq \exp(-\gamma_{\delta} \ver \log\ver p) \, ,
\end{align*}
and 
\[\Prob{\exists S\subset V, \abs{|\cut{S}| - |S|(\ver-|S|)p} > \delta \cdot |S|(\ver-|S|)p} \leq \exp(-\gamma_{\delta} \ver p) \, ,
\]
where, in both cases, $\gamma_\delta$ is a constant depending only on $\delta$.
\end{proposition}

The following proposition is similar to \cref{proposition_2_4}, but now the sets $S$ and $T$ partition the neighborhood set of a vertex in the graph. 
In particular, this result shows that, in an ER graph, the probability that there exists a vertex with a set of neighbors whose graph has many vertices or that there exists a vertex with a set of neighbors with few vertices and many edges are low.
\begin{proposition}[Proposition 2.6 of \cite{demarco2013triangle}] \label{proposition_2_6}
There is a $K > 0$ such that
\begin{equation*}
    \Prob{\exists v \in V, \exists S \subset \neighG{v} T = \neighG{v} \backslash S \text{ such that } \abs{ |\cut{S, T}| - |S||T|p} > K \ver^{3/2}p^2} \leq \exp(-\gamma_K \ver p)
\end{equation*}
where $\gamma_K$ is a constant depending only on $K$, and such that
\begin{equation*}
    \Prob{\exists v \in V, \exists S \subset \neighG{v} \text{ such that } |\spannedG{G}{S}| > \frac{|S|^2 p}{2} + K \ver^{3/2}p^2} \leq \exp(-\gamma_K \ver p)
\end{equation*}
where $\spannedG{G}{S}$ denotes the subgraph of $G$ induced by the vertices $S$ and $\gamma_K$ is a constant depending only on $K$. 
Moreover, for every  $\epsilon > 0$,
\begin{equation*}
    \Prob{\exists v \in V, \exists S \subset \neighG{v} \text{ such that } \Big[ |S| < \epsilon \ver p^2 \Big] \land  \Big[ |\spannedG{G}{S}| > \epsilon |S| \ver p^2 \Big]} \leq \exp(- \frac{\epsilon^2}{4} \ver^2 p^2) \, .
\end{equation*}
For every $\epsilon > 0$, there is a $\gamma > 0$ such that
\begin{align*}
    & \Prob{\exists v \in V, \exists S \subset \neighG{v}, T = \neighG{v} \backslash S \text{ such that } \Big[ 2 \leq |S| \leq |T| \Big] \land \Big[ |\cut{S, T}| \leq (\frac{1}{2} - \epsilon) |S| \ver p^2 \Big]} \\ 
    & \qquad \leq \exp(- \gamma \cdot \ver p^2) \, .
\end{align*}
Finally, there is a $K > 0$ and a $\gamma > 0$ such that
\begin{align*}
    & \Prob{\exists v \in V, \exists S, T \subset \neighG{v} \text{ such that } \Big[ S \cap T = \emptyset \Big] \land \Big[ |T| > \ver p/3 \Big] \land \Big[ |S| > K/p \Big] \land \Big[ |\cut{S, T}| \leq 0.9 |S||T|p \Big]} \\ &  \qquad \leq \exp(-\gamma \ver p) \, .
\end{align*}
\end{proposition}

Let \hypertarget{Q_er_prop}{$Q$} be the following event: ``all edges of $G$ are in at least one triangle''.
\begin{lemma}[Lemma 3.1 of \cite{demarco2013triangle}]
\label{lemma_3_1}
For each $\eta > 0$, there is a $\gamma_\eta > 0$ such that if $p >c \cdot \sqrt{\frac{\log\ver}{\ver}}$ with $c > 0$ big enough then
\begin{equation*}
    \Prob{\hyperlink{Q_er_prop}{Q} \land \left[ \exists F \in \Tspace^\perp \backslash \Cspace^\perp \text{ such that } |F| < \frac{1 - \eta}{4}\ver^2 p \right]} < \ver^{- \gamma_\eta c}
\end{equation*}
\end{lemma}

In the original lemma, there is no restriction on the value of $p$ (that can be as small as wanted).
We do not need such a general result, and bounding the value of $p$ makes it possible to control the value of the probability better (which is what we need).
\begin{proof}
[Sketch of the proof.]
As explained in \cite[Section 4]{demarco2013triangle}, the proof consists in showing that, if the properties of \cref{proposition_2_3,proposition_2_4,proposition_2_6}, are satisfied, then the event 
\begin{equation*}
\left\{ \hyperlink{Q_er_prop}{Q} \land \left[ \exists F \in \Tspace^\perp \backslash \Cspace^\perp \text{ such that } |F| < \frac{1 - \eta}{4}\ver^2 p \right] \right\}  
\end{equation*}
cannot occur.
From the exponential bounds of these propositions, we can show that the probability that one of the properties is not satisfied is at most $\ver^{- \gamma c}$ for some $\gamma > 0$, hence the result.
\end{proof}

The following lemma is deterministic and as such does not need adjustment: we display it here for completeness.
\begin{lemma}[Lemma 3.3 of \cite{demarco2013triangle}]
\label{lemma_3_3}
For every $\eta >0$ and $\delta > 0$, if $F \subset \fullG{\ver}$ satisfies $|F| > \frac{1 - \delta}{4} \ver^2$ and $|F \backslash \Pi| > \eta \cdot \ver^2$ for every cut $\Pi$, then for each $\epsilon > 0$, $\nbtri{F} > \frac{1}{12} (\eta - 3 \delta - \epsilon) \ver^3$.
\end{lemma}
\begin{proof}[Proof of \cref{lemma_er_beta_appendix}]
Let $c > 0$. When $p > c \cdot \sqrt{\frac{\log\ver}{\ver}}$, the event \hyperlink{Q_er_prop}{$Q$} = ``all edges of $G$ are in at least one triangle'' is very likely to happen; indeed, we can show that $\Prob{\neg \hyperlink{Q_er_prop}{Q}} < \ver^{2 - c^2}$.
Therefore, by choosing $c$ large enough, we can make the probability smaller than $\ver^{-4}$.
We can thus assume that \hyperlink{Q_er_prop}{$Q$} is satisfied, and we will then try to bound the probability that $\Tspace^\perp \neq \Cspace^\perp$ under this assumption.
Let $\eta > 0$. 
We suppose that $\Tspace^\perp \neq \Cspace^\perp$, so that $\Tspace^\perp \backslash \Cspace^\perp$ is not empty.
Because of \cref{lemma_3_1}, we can then assume that for all $F \in \Tspace^\perp \backslash \Cspace^\perp$, $|F| \geq \frac{1 - \eta}{4}\ver^2 p$.

Now, let $\vartheta = 0.1 \eta^2$.
We then choose $G$ as follows: first let $G_0 \sim G(\ver, \vartheta p)$, and then add edges of $\fullG{\ver} \backslash G_0$ independently with probability $(1 - \vartheta)p/(1 - \vartheta p)$.
Thus $G \sim G(\ver, p)$ and $G_0$ is a (small) subgraph of $G$.
Let $F_0 = F \cap G_0$.
Let $A = \{ xy \in \fullG{\ver} \backslash G_0, \neigh{G_0}{x,y} \neq \emptyset \}$, $J = \fullG{\ver} \backslash (G_0 \cup A)$ and $B = \{ xy \in A | z \in \neigh{G_0}{x,y} \implies |\{xz, yz\} \cap F_0| = 1 \}$.
Then the following equality holds:
\begin{equation*}
\label{eqGB}
    F \backslash (F_0 \cup J) = G \cap B.
\end{equation*}

Furthermore, \cref{lemma_3_1} implies that one of these must hold:
\begin{enumerate}
    \item $ |B| < \frac{1 - 2 \eta}{4} \ver^2;$
    \item There is a cut $\Pi$ such that $|B \backslash \Pi| < 0.05 \ver^2;$
    \item $\tau(B) > 0.004 \ver^3.$
\end{enumerate}

At the same time, assuming that the probabilistic properties stated in \cref{proposition_2_3,proposition_2_4,proposition_2_6} are true, $B$ must satisfy all of the following:
\begin{enumerate}[resume]
    \item $|G \cap B| > (1 - \eta - 2 \vartheta (1 + \epsilon) - \epsilon) \frac{n^2 p}{4};$
    \item For every cut $\Pi$, $|G \cap (B \backslash \Pi)| > \frac{0.4 - 2 \vartheta - \epsilon}{4} \ver^2 p;$
    \item $G \cap B$ is triangle-free.
\end{enumerate}

Since $p$ is large, $G$ contains a large number of edges with high probability, and $G \cap B$ contains a big part of the total graph $B$.
Therefore, these conditions on $B$ and $G \cap B$ are incompatible.
Indeed, we can show that:
\begin{enumerate}[resume]
\item $\Prob{\exists F_0 \subset G_0 \text{ such that } |B| < \frac{1 - 2 \eta}{4}} \leq \exp(- 0.02 \eta^2 \cdot \ver^2 p);$
\item $\Prob{\exists F_0 \subset G_0, \exists \text{ a cut } \Pi \text{ such that } |B \backslash \Pi| < 0.05 \ver^2} \leq \exp(- 0.05 \ver^2 p) \text{ when $\eta$ is small;}$
\item $\Prob{\exists F_0 \subset G_0 \text{ such that } \nbtri{B} > 0.004 \ver^3 \land G \cap B \text{ triangle-free}} < \exp (- (0.0006 - \vartheta) \ver^2 p)$.
\end{enumerate}

If $c$ is large enough, then for every $p > c \cdot \sqrt{\frac{\log\ver}{\ver}}$, all these probabilities are $\mathcal{O}(\ver^{-4})$.
This is precisely  what we want: the probability that one of the assumptions we made so far is not satisfied is $\mathcal{O}(\ver^{-4})$, and if all the assumptions we made are satisfied, then the only scenario where $\Tspace^\perp \neq \Cspace^\perp$ have a probability $\mathcal{O}(\ver^{-4})$ to occur.
Hence, $\Prob{\Tspace^\perp \neq \Cspace^\perp} = \mathcal{O}(\ver^{-4})$.
\end{proof}

\end{document}